\documentclass[a4paper,11pt,reqno]{amsart}

\usepackage{times}
\usepackage{bbm}

\usepackage{geometry}
\usepackage{amsmath,amsfonts,amssymb}
\usepackage{mathtools}
\usepackage{array}

\usepackage{amsthm}
\usepackage{thmtools}

\usepackage[dvipsnames]{xcolor}

\usepackage{graphicx}
\usepackage{subfig}
\usepackage{caption}
\usepackage{tabularx}
\usepackage{float}
\usepackage{tikz}

\usepackage{enumitem}

\mathtoolsset{showonlyrefs=true}

\newtheorem{theorem}{Theorem}[section]
\newtheorem{proposition}[theorem]{Proposition}
\newtheorem{lemma}[theorem]{Lemma}
\newtheorem{corollary}[theorem]{Corollary}

\theoremstyle{definition}
\newtheorem{definition}[theorem]{Definition}

\newtheorem{remark}[theorem]{Remark}

\newlength{\assumptiontagwidth}
\makeatletter

\NewDocumentEnvironment{assumption}{m +b}
{%
  \par
  \addvspace{.5\baselineskip}%
  \noindent
  \phantomsection
  \def\@currentlabel{#1}%
  \parbox[c]{\assumptiontagwidth}{%
    \centering
    \mbox{\normalfont (#1)}%
  }%
  \hspace{1em}%
  \parbox[c]{\dimexpr\linewidth-\assumptiontagwidth-1em\relax}{%
    \slshape
    \setlength{\parindent}{0pt}%
    #2%
  }%
  \par
  \addvspace{.5\baselineskip}%
}
{}

\makeatother
\numberwithin{equation}{section}

\usepackage[square,numbers,sort&compress]{natbib}

\usepackage[
    colorlinks=true,
    linkcolor=blue,
    anchorcolor=black,
    citecolor=blue,
    filecolor=black,
    menucolor=black,
    runcolor=black,
    urlcolor=black
]{hyperref}

\newcommand{\e}{\mathrm e}
\newcommand{\I}{\mathrm i}

\newcommand{\cA}{\mathcal A}
\newcommand{\cB}{\mathcal B}
\newcommand{\cC}{\mathcal C}

\newcommand{\cE}{\mathcal E}
\newcommand{\cF}{\mathcal F}

\newcommand{\cJ}{\mathcal J}

\newcommand{\cL}{\mathcal L}

\newcommand{\cN}{\mathcal N}

\newcommand{\cQ}{\mathcal Q}

\newcommand{\cU}{\mathcal U}

\newcommand{\rmT}{\mathrm T}

\newcommand{\NN}{\mathbb{N}}
\newcommand{\ZZ}{\mathbb{Z}}
\newcommand{\R}{\mathbb{R}}

\DeclareMathOperator{\supp}{{\rm supp}}
\DeclareMathOperator{\sign}{{\rm sign}}

\newcommand{\eps}{\varepsilon}
\newcommand{\de}{\partial}

\newcommand{\dd}{\mathop{}\!\mathrm{d}}
\newcommand{\ddt}{\frac{\dd}{\dd t}}

\newcommand{\norm}[1]{\left\lVert #1 \right\rVert}

\makeatletter

\def\@tocline#1#2#3#4#5#6#7{\relax
  \ifnum #1>\c@tocdepth % then omit
  \else
    \par \addpenalty\@secpenalty\addvspace{#2}%
    \begingroup \hyphenpenalty\@M
    \@ifempty{#4}{%
      \@tempdima\csname r@tocindent\number#1\endcsname\relax
    }{%
      \@tempdima#4\relax
    }%
    \parindent\z@ \leftskip#3\relax \advance\leftskip\@tempdima\relax
    \rightskip\@pnumwidth plus4em \parfillskip-\@pnumwidth
    #5\leavevmode\hskip-\@tempdima
      \ifcase #1
       \or\or \hskip 1em \or \hskip 2em \else \hskip 3em \fi%
      #6\nobreak\relax
    \dotfill\hbox to\@pnumwidth{\@tocpagenum{#7}}\par
    \nobreak
    \endgroup
  \fi}
\makeatother

\makeatletter

\def\@settitle{%
  \begin{center}%
    \baselineskip14\p@\relax
    \large
    \bfseries
    \uppercasenonmath\@title
    \@title
  \end{center}%
}

\def\@setauthors{%
  \begingroup
  \def\thanks{\protect\thanks@warning}%
  \trivlist
  \centering
  \normalsize
  \@topsep30\p@\relax
  \advance\@topsep by -\baselineskip
  \item\relax
  \author@andify\authors
  \def\\{\protect\linebreak}%
  \MakeUppercase{\authors}%
  \ifx\@emptyOntribs
  \else
    ,\penalty-3 \space \@setcontribs
    \@closetoccontribs
  \fi
  \endtrivlist
  \endgroup
}

\def\section{\@startsection{section}{1}%
  \z@
  {.5\linespacing\@plus\linespacing}
  {.5\linespacing}
  {\normalfont\centering\scshape\large}}

\def\subsection{\@startsection{subsection}{2}%
  \z@
  {.5\baselineskip\@plus.2\baselineskip\@minus.1\baselineskip}
  {-0.5em}
  {\normalfont\scshape}}
  
\def\subsubsection{\@startsection{subsubsection}{3}%
  \z@
  {.4\baselineskip\@plus.2\baselineskip\@minus.1\baselineskip}
  {-0.5em}
  {\normalfont\itshape}}

\def\paragraph{\@startsection{paragraph}{4}%
  \parindent
  {.4\baselineskip\@plus.2\baselineskip\@minus.1\baselineskip}
  {-0.5em}
  {\normalfont\itshape}}

\makeatother

\title[2D3C rotating stratified Boussinesq]{Extended stability timescales for the 2D3C \\rotating stratified Boussinesq equations}
\author{Augusto Del Zotto}
\address{Institute of Mathematics, University of Zurich}
\email{augusto.delzotto@math.uzh.ch}
\author{Luca Melzi}
\address{Department of Mathematics, Imperial College London}
\email{l.melzi24@imperial.ac.uk}
\author{Haoling Xiang}
\address{Department of Mathematics, Imperial College London}
\email{h.xiang23@imperial.ac.uk}
\date{\today}


\begin{document}

\begin{abstract}
We study the time evolution of a vertical-slice model for incompressible inhomogeneous rotating fluids perturbed around a stably stratified rest state with linear buoyancy profile.
For $\eps$-small localized initial data with vanishing potential vorticity, we prove that solutions remain of size $O(\eps)$ up to times of order $\eps^{-\frac43}$, and up to times of order $\eps^{-2}$ in a regime where the Brunt--V\"ais\"al\"a frequency and the Coriolis parameter are comparable.
The first result relies on linear dispersion with amplitude decay rate $\langle t\rangle^{-\frac12}$, while the second exploits the absence of three-wave resonances via a normal form transformation.
Our results remain valid for small nonzero initial potential vorticity, whose size is propagated by means of an intrinsic change of coordinates.
\end{abstract}

\maketitle

\tableofcontents

\section{Introduction and main results}
The dynamics of atmospheric fronts and coastal oceanic currents is shaped by the combined action of the slow variation in one horizontal direction and of two restoring mechanisms, namely the Coriolis and the buoyancy forces \cite{Pedlosky1987,HoskinsBretherton1972,CotHol2013}.
In the ocean interior, the same two mechanisms sustain inertia-gravity waves whose propagation and nonlinear interactions play a central role in the energy budget and interior mixing of the ocean \cite{MunkWunsch1998}.

The model considered in these pages, which we refer to as the two-dimensional, three-component (2D3C) rotating stratified Boussinesq system, is the vertical slice model introduced in \cite{XieBuhler2019,ShavitBuhlerShatah2026} to study inertia-gravity wave turbulence.
It is obtained from the three-dimensional rotating Boussinesq system, written for perturbations of a linearly stratified fluid at rest, by assuming that all unknowns are independent of $x_2$, the second horizontal coordinate.\footnote{Since the pressure is then also independent of $x_2$, the set of $x_2$-independent solutions is an invariant subspace of the three-dimensional dynamics.}
Explicitly, it reads
\begin{equation}\label{eq:2D3C}
\begin{cases}
    \de_t u+(u\cdot\nabla) u+fe_3\times u+\nabla p=be_3,\\
    \de_t b+u\cdot\nabla b+N^2u_3=0,\\
    \nabla\cdot u=0,
\end{cases}
\tag{2D3C}
\end{equation}
where all unknowns depend only on $(t,x_1,x_3)\in \R^3$ and $\nabla=(\de_1,0,\de_3)$ (with a slight abuse of notation, we use the same symbol for $\nabla=(\de_1,\de_3)$).
Here $e_3$ is the upward unit vector, $u=(u_1,u_2,u_3)$ the velocity field, $p(t,x_1,x_3)$ the pressure, $N^2x_3+b(t,x_1,x_3)$ the total buoyancy (hence $b$ is the buoyancy perturbation), $f\neq0$ the Coriolis parameter and $N\neq0$ the Brunt--V\"ais\"al\"a frequency, both assumed constant.

In this work we investigate the long-time stability of the rest state of \eqref{eq:2D3C}, exploiting the dispersive effects that characterize rotating and stratified fluids \cite{CheminDesjardinsGallagherGrenier2006, Majda2003}.
These are appreciated in the model under scrutiny upon linearizing it and observing that the dynamics splits into one zero-frequency mode, representing the linear potential vorticity of the system, and two inertia-gravity wave modes.
On one hand, the two wave modes carry the dispersive effects and, when $|N|\neq|f|$, disperse with frequencies $\pm\Omega$, where
\begin{equation} \label{eq:dispersion_relation_intro}
    \Omega(\xi)=\sqrt{\frac{N^2\xi_1^2+f^2\xi_3^2}{\xi_1^2+\xi_3^2}},
\end{equation}
along the horizontal and vertical frequencies $\xi_1$ and $\xi_3$.
On the other hand, the zero-frequency mode is invariant under the linear dynamics. 
This steers the nonlinear analysis towards a natural well-prepared initial data condition obtained by imposing a vanishing potential vorticity constraint, the nonlinear counterpart of what assumed in \cite[Section~2.3]{ShavitBuhlerShatah2026}.

Our main results show that, for well-prepared initial data and under suitable conditions on $f$ and $N$, the stability timescale of  solutions of \eqref{eq:2D3C} can be extended beyond the classical local well-posedness $O(\eps^{-1})$ existence time for $\eps$-small initial data \cite{MajdaBertozzi2001}.

\begin{theorem} \label{thm:43}
    Fix $N,f\in \R\setminus\{0\}$ such that $N^2\neq f^2$ and $s>4$.
    Then, there exist constants $\eps_0>0$, $c_0>0$, and $\Lambda_{N,f}>0$, depending only on $N$, $f$, and $s$, such that the following holds.
    If the initial datum $(u(0),b(0))$ is divergence free, well-prepared in the sense of \eqref{ass:wp}, and satisfies
    \begin{equation}\label{eq:thm1_assumption_ic_0}
        \norm{(u(0),|N|^{-1}b(0))}_{H^{s}\cap \dot B^{3}_{1,1}}=\eps\leq \eps_0,
    \end{equation}
    then, for $T_\eps:= c_0 \Lambda_{N,f}  \eps^{-\frac43}$, the unique  solution of \eqref{eq:2D3C} satisfies $(u,b)\in \cC([0,T_\eps];H^s(\R^2)^4)$ with
    \begin{equation}\label{eq:extended_stability}
         \sup_{t\in[0,T_\eps]}\norm{(u(t),|N|^{-1}b(t))}_{H^s}  \leq \sqrt{\frac32} \eps.
    \end{equation}
\end{theorem}

Additionally, when the Coriolis parameter and the Brunt--V\"ais\"al\"a frequency are quantitatively comparable, the stability timescale of solutions can be extended further.

\begin{theorem} \label{thm:2}
    Let $s>5$ and assume, in addition to the hypotheses of Theorem~\ref{thm:43}, that
    \begin{equation}\label{eq:nonresonant_regime}
        \max\{|N|,|f|\}<2\min\{|N|,|f|\}.
    \end{equation}
    Then, there exist constants $\tilde \eps_0>0$, $\tilde c_0>0$, and $\tilde \Lambda_{N,f}>0$, depending only on $N$, $f$, and $s$, such that, if the initial datum is divergence free, well-prepared in the sense of \eqref{ass:wp}, and satisfies \eqref{eq:thm1_assumption_ic_0} with $\eps\leq \tilde\eps_0$, the unique solution $(u(t),b(t))$ to \eqref{eq:2D3C} satisfies \eqref{eq:extended_stability} with
    \begin{equation}\label{eq:thm2span}
         T_\eps := \tilde c_0\tilde \Lambda_{N,f}  \eps^{-2}.
    \end{equation}
\end{theorem}
Theorems~\ref{thm:43} and~\ref{thm:2} are the rotating counterpart of the $\eps^{-\frac43}$ and $\eps^{-2}$ timescales obtained in \cite{ElgWid2015} and \cite{JurjaWidmayer2026}, respectively, for the two-dimensional Boussinesq system without rotation. Note that the localization assumptions of \cite{JurjaWidmayer2026} differ from ours.
In three dimensions and without stratification, the analogous fixed-rotation problem for the Euler--Coriolis equations has been studied in \cite{GuoHuangPausaderWidmayer2023}, where sufficiently small axisymmetric data are shown to have lifespans of at least $\eps^{-M}$ for every $M>0$, followed by global axisymmetric solutions \cite{GuoPauWid2023} and by global solutions for suitably localized data without axisymmetry \cite{RenTian2024}.

To the best of the authors' knowledge, Theorems~\ref{thm:43} and~\ref{thm:2} are the first quantitative extensions of the stability timescales for  solutions of \eqref{eq:2D3C}.

The remainder of the section collects some comments about the statements of our results, discusses how they relate to previous work, and presents the main ideas of the proofs.

\subsection{About the statements} \label{subsec:about_the_theorems}
We start by breaking down the content of Theorem~\ref{thm:43} and Theorem~\ref{thm:2}, analyzing and motivating the presence of important parts and drawing parallels with the existing results.

\paragraph{Well-prepared data and potential vorticity.} 
A standard intrinsic quantity of \eqref{eq:2D3C}, and more generally of rotating and stratified flows, is the potential vorticity \cite{Ertel1942}.\footnote{See \cite{SchubertEtAl2004} for an English translation of Ertel's papers on potential vorticity.}
Physically, it measures the component of the absolute vorticity $\varpi+fe_3$ (where $\varpi=\nabla\times u$) normal to the surfaces of constant total buoyancy $N^2x_3+b$, and thus couples the two restoring mechanisms of the system: rotation enters through the vorticity, stratification through the buoyancy gradient \cite{HoskinsMcIntyreRobertson1985}.
In the present setting, its deviation from the rest-state value $fN^2$, multiplied by $f$, reads
\begin{equation}\label{eq:potential_vorticity}
    q:=N^2f\de_1u_2+f^2\de_3b+fJ(u_2,b),
\end{equation}
where $J(u_2,b):=\de_1u_2\de_3b-\de_3u_2\de_1b$.
Our fundamental assumption is the following.
\begin{assumption}{wp}\label{ass:wp}
    We say that the initial datum $(u(0),b(0))$ is \emph{well-prepared} if the initial potential vorticity satisfies $q(0)=0$.
\end{assumption}
Crucially, by Ertel's theorem, the potential vorticity is transported by the flow,\footnote{Recall that $u\cdot \nabla=u_1\de_1+u_3\de_3$ since all functions are $x_2$-independent.}
\begin{equation}\label{eq:transport_q}
    (\de_t+u_1\de_1+u_3\de_3)q=0,
\end{equation}
hence a vanishing initial vorticity remains zero for all times.
\begin{remark} \label{rmk:gamma_well_prepared}
    As shown in Section~\ref{subsec:PV}, the transport structure of \eqref{eq:transport_q} allows one to propagate any smallness condition imposed on $q$ for as long as the two quantities $(u_2,b)$ remain small enough.
    Our main results are also valid under the following relaxed well-preparedness assumption.
    \begin{assumption}{$\gamma$-wp}\label{ass:gamma_wp}
        Fix $0<\eps<1$ and $\gamma>0$. We say that the initial datum $(u(0),b(0))$ is \textit{$\gamma$-well-prepared} if the initial potential vorticity satisfies $\norm{q(0)}_{H^{s-1}}\leq \eps^\gamma$, where $s$ is the regularity given in Theorems~\ref{thm:43} and~\ref{thm:2}.
    \end{assumption}
    It is clear that, if the initial datum is $\gamma$-well-prepared for all $\gamma>0$, then $q(0)=0$.
    Moreover, the Sobolev regularity $s-1$ reflects the fact that the potential vorticity is, by definition, one derivative less regular than $u$ and $b$ (cf. \eqref{eq:potential_vorticity}).
    Regarding the choice of $\gamma$, to recover the $O(\eps^{-\frac43})$ or $O(\eps^{-2})$ timescales it suffices to assume $\gamma\geq\frac43$ or $\gamma\geq 2$, respectively.
\end{remark}

\paragraph{Parameter dependence.}

Our main results concern fixed values of $N$ and $f$, and are of a different nature from the ones for strong rotation or stratification regimes.
To connect Theorems~\ref{thm:43} and~\ref{thm:2} with the broader literature, the constant $\Lambda_{N,f}$ can be explicitly computed in terms of the parameters $N,f$, and it reads 
\begin{equation} \label{eq:Lambda_Nf}
    \Lambda_{N,f}=c_s\big||N|-|f|\big|^{\frac13}\left(\frac{\min\{|N|,|f|\}}{\max\{|N|,|f|\}}\right)^4,
\end{equation}
where $c_s>0$ is a constant depending only on the Sobolev regularity $s$.
We refer the reader to Remark~\ref{rmk:Lambda_receipt} for details on how the constant is tracked.
Once a joint scaling of $N$ and $f$ is fixed, we see that \eqref{eq:Lambda_Nf} allows one to determine, for any given time $T>0$, values of the parameters for which the solution is $\varepsilon$-small on the time horizon $[0,T]$.
\begin{remark}
    The ratio term in \eqref{eq:Lambda_Nf} appears to be an artifact of our method. 
    We expect that a different approach could improve the constant $\Lambda_{N,f}$ by removing this factor, which would allow to take the limits $N\to\infty$ and $f\to\infty$ independently.
\end{remark}
In the fixed-ratio regime, a direct comparison can be drawn with results from the three-dimensional rotating and stratified Boussinesq equations obtained in \cite{EmbidMajda1996,BabinEtAl1997,BabinEtAl1999} through fast-wave averaging and singular limits.
For the Euler--Coriolis equations, lifespan estimates at large rotation speed are obtained in \cite{Dutrifoy2005,KohLeeTakada2014,Takada2016}, while for the three-dimensional stratified Boussinesq equations without rotation, existence on any prescribed time interval for sufficiently large $N$, together with the strongly stratified limit, is proven in \cite{Takada2019}.
In the two-dimensional stratified case, which is the closest to our setting, long-time existence for large $N$ is obtained in \cite{Takada2021,AnguloFerreiraKosloff2024}.

Let us briefly comment on the physical relevance of the regimes covered by the analysis in Theorems~\ref{thm:43} and~\ref{thm:2}.
In this regard, we mention that the typical oceanic regime has $|N/f|\ge40$ (see \cite{ShavitBuhlerShatah2026}), whereas in the deep ocean stratification is much weaker, so that $|N/f|\sim 1$ and the two effects are of comparable strength  (see \cite{vanMil2006}).

Finally, we remark that the constant $\tilde \Lambda_{N,f}$ appearing in Theorem~\ref{thm:2} could also be tracked explicitly, but we do not pursue this here.

\paragraph{Regularity and initial data localization.} 
In addition to the well-preparedness constraint, we assume $(u(0),b(0))\in H^s\cap \dot B^3_{1,1}(\R^2)$. 
This choice is classical and is motivated as follows. 
\begin{enumerate}[label=-]
    \item The Sobolev space $H^s$ is the natural energy space in which stability problems are analyzed.
    The requirement $s>4$ (resp. $s>5$) is dictated by the loss of two derivatives in the dispersive estimate, see \eqref{eq:general_dispersive_estimate}, and by the loss of one derivative in each application of the nonlinearity $u\cdot\nabla$, which accumulates over two (resp. three) iterations. 
    \item The homogeneous Besov space $\dot B^3_{1,1}$ is a natural space in which to capture dispersive effects.
    The initial data are assumed to be localized in an $L^1$-based sense, which is what a linear $L^1-L^\infty$ dispersive estimate requires, and the derivatives lost when obtaining the corresponding time decay of the linear propagator are reflected in the regularity index.
\end{enumerate}
Localization assumptions of this type go back to \cite{ElgWid2015,Widmayer2018}, where initial data are assumed in $W^{s,1}$.

\begin{remark}
    Theorem~\ref{thm:43} admits a further relaxation of its hypotheses, in the spirit of \cite{Wan2020,JujKo2025}: the Besov localization of the initial datum can be dropped, leading to a purely $H^s$-based statement.
    Indeed, the dispersive estimate established later in Proposition~\ref{prop:general_dispersive_estimate} yields, through a standard $TT^*$ argument, Strichartz estimates analogous to \cite[Theorem 3.4]{JujKo2025}, which control the linear evolution of the initial datum in terms of its $H^s$ norm alone.
    This approach is still sufficient for reaching the timescale of Theorem~\ref{thm:43}, but not for the one of Theorem~\ref{thm:2}.
\end{remark}

\subsection{About the proofs}\label{sec:about_proofs}
Translating dispersion into quantitative stability estimates is a recurring theme in the analysis of systems where stratification or rotation are present.
We refer the reader to \cite{CheminDesjardinsGallagherGrenier2006,Majda2003} for a general discussion, to \cite{ElgWid2015,JurjaWidmayer2026,JujKo2025} for the two-dimensional stratified case, and to \cite{KohLeeTakada2014,GuoHuangPausaderWidmayer2023,GuoPauWid2023} for rotating flows.
Our proofs combine ideas from \cite{ElgWid2015} for the stability timescale of Theorem~\ref{thm:43}, while for Theorem~\ref{thm:2} we resort to the normal form method introduced in \cite{Shatah1985}, a general discussion can be found in \cite{Ger10}.

In what follows, by writing $A\lesssim B$ we make implicit the existence of a constant $C>0$, depending only on the data of the problem, such that $A\leq CB$.
Moreover, by $A\sim B$ we mean that both $A\lesssim B$ and $B\lesssim A$.
We may write $\lesssim_{\eta}$ or $\sim_{\eta}$ to specify that the hidden constant depends on a specific parameter $\eta$, and the same convention applies to $O_\eta(\cdot)$.

To pinpoint better where our proofs depart from the standard theory, we summarize briefly the main line of work followed in the manuscript.
Both proofs combine an $H^s$ energy bootstrap with a blow-up criterion of Beale--Kato--Majda type, which together reduce the problem to controlling the time integral of 
\begin{equation}
    \norm{\nabla u(t)}_{L^\infty}+|N|^{-1}\norm{\nabla b(t)}_{L^\infty}.
\end{equation}
To estimate this quantity, we study the system in the variables $(q,z_+,z_-)$, namely the potential vorticity \eqref{eq:potential_vorticity} and the two wave variables arising from the diagonalization of the linearized system. 
The well-preparedness assumption \eqref{ass:wp} removes the dynamics of $q$, and the analysis reduces to the evolution of the two wave variables $z_\pm$, which satisfy
\begin{equation}
    \de_tz_\pm\mp\I\Omega(D)z_\pm = \text{nonlinear terms}.
\end{equation}
For $N^2\neq f^2$, the linear semigroups   $\e^{\pm\I\Omega(D)t}$ exhibit the sharp $\langle t\rangle^{-\frac12}$ dispersive decay.
This is proved by means of a Littlewood--Paley decomposition and the van der Corput lemma in polar coordinates, in the spirit of \cite{GuoPauWid2023,JurjaWidmayer2026,GuoHuangPausaderWidmayer2023}.

The two proofs diverge at this point.
Theorem~\ref{thm:43} is obtained by applying the dispersive estimate to $O(\eps)$ initial data and $O(\eps^2)$ nonlinear terms, heuristically obtaining
\begin{equation}\label{eq:heurist43}
    \norm{\nabla u(t)}_{L^\infty}+|N|^{-1}\norm{\nabla b(t)}_{L^\infty}\lesssim_{N,f} \langle t\rangle^{-\frac12}\eps+\langle t\rangle^{\frac12}\eps^2,
\end{equation}
whose time integral is of order one up to time $T_\eps\sim_{N,f}\eps^{-\frac43}$.
\begin{remark}\label{rmk:Lambda_receipt}
    The explicit expression \eqref{eq:Lambda_Nf} of $\Lambda_{N,f}$ can be obtained from \eqref{eq:heurist43} by carefully tracking the constants multiplying the $\eps^2$ term.
    For ease of presentation and to keep the focus of the manuscript on the $\eps$-dependence of the stability timescales, we only comment on it here.
    In particular, for some constant $C_s>0$ depending only on the Sobolev regularity, integrating \eqref{eq:heurist43} and keeping track of the implicit dependence on $N,f$, one obtains
    \begin{equation}
        \int_0^{T_\eps}\norm{\nabla u(t)}_{L^\infty}+|N|^{-1}\norm{\nabla b(t)}_{L^\infty}\dd t\leq C_s\big||N|-|f|\big|^{-\frac12}\left(\frac{\max\{|N|,|f|\}}{\min\{|N|,|f|\}}\right)^{6}T^{\frac32}_\eps\eps^2+O_{N,f}(T_{\eps}^{\frac12}\varepsilon),
    \end{equation}
    where the constant multiplying $T^{\frac32}_\eps\eps^2$ arises as follows.
    First, the dispersion constant $\big||N|-|f|\big|^{-\frac12}$ appears in Lemma~\ref{lemma:local_dispersion}.
    Then, using the change of variables \eqref{eq:D_bound_change_variables} twice, we obtain the contribution $\max\{|N|,|f|\}^4\min\{|N|,|f|\}^{-6}$ from the norm of the multiplier $\Omega^{-1}$, tracked in Proposition~\ref{prop:omega-1_bound}.
    Finally, the change of variables \eqref{eq:D_bound_change_variables} contributes an additional factor $\max\{|N|,|f|\}^2$.
\end{remark}
To further extend the $O_{N,f}(\eps^{-\frac43})$ timescale, we derive the equation for the profiles $\e^{\mp\I\Omega(D)t} z_\pm(t)$, in which the quadratic wave-wave interactions oscillate with phases
\begin{equation}
    \Phi_{\sigma,\nu\mu}(\xi,\eta):=-\sigma\Omega(\xi)+\nu\Omega(\eta-\xi)+\mu\Omega(\eta),\qquad \sigma,\nu,\mu\in\{\pm\}.
\end{equation}
As noted in \cite{RemmelSukhatmeSmith2010}, in the regime of Theorem~\ref{thm:2} there are no three-wave resonances, and in fact 
\begin{equation}
    |\Phi_{\sigma,\nu\mu}(\xi,\eta)|\geq 2\min\{|N|,|f|\}-\max\{|N|,|f|\}>0,
\end{equation}
for all $\xi,\eta\in\R^2$. 
This allows for an integration by parts in time (see \cite{Ger10}), which replaces the $O(\eps^2)$ nonlinear contribution with quadratic boundary terms and cubic nonlinear contributions, improving the above heuristic bound \eqref{eq:heurist43} to
\begin{equation} \label{eq:heurist2}
    \norm{\nabla u(t)}_{L^\infty}+|N|^{-1}\norm{\nabla b(t)}_{L^\infty}\lesssim _{N,f}\langle t\rangle^{-\frac12}\eps+\eps^2+\langle t\rangle^{\frac12}\eps^3.
\end{equation}
Then, the time integral of \eqref{eq:heurist2} remains of order one up to time $T_\eps\sim_{N,f}\eps^{-2}$.

Under the $\gamma$-well-prepared assumption \eqref{ass:gamma_wp}, the additional contributions coming from the potential vorticity are controlled by propagating its initial $\eps^\gamma$-smallness, so that the bootstrap argument applies to the same timescales.
\begin{remark}
Two ingredients of the proof deserve some comments.
First, the propagation of $H^s$-smallness of the potential vorticity \eqref{eq:potential_vorticity} relies on a second pair of transported quantities, already pointed out in \cite{XieBuhler2019,ShavitBuhlerShatah2026}: the vertical stratification variables and the (rescaled) absolute momentum, respectively, $\Psi_1:=N^2x_3+b$ and $\Psi_2:=f^2x_1+fu_2$.
As long as $(u_2,b)$ remains small, these two transported quantities combine into a change of coordinates close to the identity, which is used in Proposition~\ref{proposition:q_propagation} to bound $\norm{q(t)}_{H^{s-1}}$ by $\norm{q(0)}_{H^{s-1}}$, hence propagating the desired initial smallness. 

Second, the pseudo-product operators arising from the normal form transformation in Section~\ref{sec:proofof2} are bounded using the elementary representation 
\begin{equation} \label{eq:integral_Laplace}
    \frac{1}{|\Phi_{\sigma,\nu\mu}(\xi,\eta)|}=\int_0^\infty \e^{\kappa\sigma\Omega(\xi)\rho}\e^{-\kappa\nu\Omega(\xi-\eta)\rho}\e^{-\kappa\mu\Omega(\eta)\rho}\dd\rho, \qquad \kappa:=\sign{\Phi_{\sigma,\nu\mu}},
\end{equation}
which factorizes the symbol $\Phi^{-1}_{\sigma,\nu\mu}$ into products of Fourier multipliers, each acting on a single frequency variable.
This is possible since $\Omega$ is bounded and, in the non-resonant regime \eqref{eq:nonresonant_regime}, the integral in \eqref{eq:integral_Laplace} converges.
This substantially simplifies the control of the nonlinear contributions. 
\end{remark}

\subsection{Outline of the paper}
The exposition is organized as follows.
Section~\ref{sec:linear_dynamics} is devoted to the study of the linearized problem: we prove in Proposition~\ref{prop:general_dispersive_estimate} the sharp dispersive estimate in homogeneous Besov spaces associated to the dispersion relation \eqref{eq:dispersion_relation_intro}.
In Section~\ref{sec:nonlinear_dynamics} we set up the bootstrap argument for the full problem \eqref{eq:2D3C}, from which Theorems~\ref{thm:43} and~\ref{thm:2} follow.
In Sections~\ref{sec:proofof43} and~\ref{sec:proofof2} we close the bootstrap for the timescales $T_\eps\sim_{N,f}\eps^{-\frac43}$ and $T_\eps\sim_{N,f}\eps^{-2}$, respectively.
Appendix~\ref{app:appendixBesov} contains the definition of homogeneous Besov spaces and useful properties used in this work.

\section{Linear dynamics} \label{sec:linear_dynamics}

The goal of this section is to investigate the linear dynamics of the system \eqref{eq:2D3C}.
The strategy is to diagonalize the system written in Fourier variables, thereby exposing the oscillatory modes and their dispersion relation \eqref{eq:dispersion_relation_intro}.

To do so, we start by introducing the scalar vorticity $\omega=\nabla^\perp\cdot U$, where $\nabla^\perp=(\de_3,-\de_1)$ and $U:=(u_1,u_3)$, in terms of which the linearization of \eqref{eq:2D3C} reads
\begin{equation}\label{eq:full_system_final_linearised}
    \begin{cases}
    \de_t\omega =\de_3a -\de_1b ,\\
    \de_ta =f^2(-\Delta)^{-1}\de_3\omega ,\\
    \de_tb =-N^2(-\Delta)^{-1}\de_1\omega ,
    \end{cases}
\end{equation}
where $a:=fu_2$.
It is convenient to write it in Fourier variables.
We adopt the following convention: for any function $\phi$, its Fourier transform and its inverse are given by
\begin{equation}
    \hat{\phi}(\xi)=\int_{\R^2}\phi(x)\e^{-\mathrm{i}x\cdot\xi}\dd x\quad  \longleftrightarrow \quad\phi(x)=\frac{1}{4\pi^2}\int_{\R^2}\hat\phi(\xi)\e^{\mathrm{i}x\cdot\xi}\dd\xi,
\end{equation}
where $x=(x_1,x_3)\in\R^2$ and $\xi=(\xi_1,\xi_3)\in\R^2$.

This yields the following pointwise-in-frequency autonomous ODE system
\begin{equation}
\de_t\begin{bmatrix}
\hat{\omega} (\xi) \\
\hat{a} (\xi) \\
\hat{b} (\xi)
\end{bmatrix}=\begin{bmatrix}
0 & \mathrm{i}\xi_3 & -\mathrm{i}\xi_1 \\
\mathrm{i}f^2\frac{\xi_3}{|\xi|^2} & 0 & 0 \\
-\mathrm{i}N^2\frac{\xi_1}{|\xi|^2} & 0 & 0
\end{bmatrix}\begin{bmatrix}
\hat{\omega} (\xi) \\
\hat{a} (\xi) \\
\hat{b} (\xi)
\end{bmatrix},
\end{equation}
which can be solved explicitly.
Indeed, the eigenvalues of the above matrix are
\begin{equation}\label{eq:eigenvalues}
    \lambda_0(\xi)=0, \quad \lambda_+(\xi)=\mathrm{i}\Omega(\xi),\quad \lambda_-(\xi)=-\mathrm{i}\Omega(\xi),
\end{equation}
where the positive symbol $\Omega$ represents the dispersion relation and reads
\begin{equation}\label{eq:dispersion_relation_fourier}
    \Omega(\xi):=\sqrt{\frac{N^2\xi_1^2+f^2\xi_3^2}{\xi_1^2+\xi_3^2}}.
\end{equation}
The corresponding diagonal variables are
\begin{equation}\label{eq:eigenfunctions_fourier}
    \hat q_0(\xi)=\I N^2\xi_1\hat a(\xi)+\I f^2\xi_3 \hat b(\xi),\quad \hat z_\pm(\xi)=\frac1{2|\xi|} \left(\hat \omega(\xi)\pm \Omega^{-1}(\xi)\left(\xi_3\hat a(\xi)-\xi_1\hat b(\xi)\right)\right),
\end{equation}
and satisfy the equations 
\begin{equation}\label{eq:for_hatqz+z-_linear}
    \de_t\hat q_0(\xi)=0, \qquad \de_t\hat z_\pm(\xi)=\pm\I\Omega(\xi)\hat z_\pm(\xi).
\end{equation}
\begin{remark}
    It is natural to associate to these variables their physical-space counterparts, namely 
    \begin{equation}\label{eq:eigenfunctions_real}
     q_0(x)= N^2\de_1 a(x)+f^2\de_3  b(x),\quad  z_\pm(x)=\frac1{2|\nabla|} \left(\omega(x)\mp \I \Omega^{-1}(D)\left(\de_3 a(x)-\de_1 b(x)\right)\right),
    \end{equation}
    where 
    \begin{equation}\label{eq:dispersion_relation_real}
        \Omega(D):=\sqrt{\frac{N^2\de_1^2+f^2\de_3^2}{\de_1^2+\de_3^2}}.
    \end{equation}
    
    We immediately notice that $q_0$ is the linearization of the potential vorticity $q$ defined in \eqref{eq:potential_vorticity}. 
    The variables $z_\pm$ are the two wave variables that complete the descriptive picture of the system \eqref{eq:2D3C}, and the ones that offer rich oscillatory dynamics.
    We note that, upon writing $\omega$ in terms of $U$, they can also be written as
    \begin{equation}
        z_\pm(x)=\frac{1}{2}\left(\left(\frac{\nabla}{|\nabla|}\right)^\perp\cdot U(x)\mp\mathrm{i}\Omega^{-1}(D)\left(\frac{\de_3}{|\nabla|}a(x)-\frac{\de_1}{|\nabla|}b(x)\right)\right).
    \end{equation}
\end{remark}

\subsection{Linear dispersion}
To exploit the oscillatory structure appearing in the equations \eqref{eq:for_hatqz+z-_linear} for $\hat z_\pm$, we analyze the dispersion relation $\Omega$ and establish decay estimates for the semigroup associated with \eqref{eq:dispersion_relation_fourier}.
This is contained in Proposition~\ref{prop:general_dispersive_estimate}, while Proposition~\ref{prop:sharp_dispersion} shows that the time decay is sharp.

We start by writing the action of the semigroup generated by $\I\sigma\Omega(D)$, where $\sigma\in\{\pm\}$, on smooth functions $\phi$ with $\hat\phi\in\mathcal{C}^\infty_c(\R^2)$ as
\begin{equation}
    \e^{\I\sigma\Omega(D) t}\phi(x):=\frac1{4\pi^2}\int_{\R^2}\e^{\I x\cdot\xi+\I\sigma\Omega(\xi)t}\hat{\phi}(\xi)\dd\xi,
\end{equation}
and we establish a frequency-localized dispersive estimate.
\begin{lemma}\label{lemma:local_dispersion}
    Fix $\sigma\in\{\pm\}$ and assume $Nf\neq 0$. Let $\hat\phi$ be a smooth function supported on $\cU=\{\xi:|\xi|\in[\frac34,\frac83]\}$.
    Then, there exists a constant $C>0$, depending only on $\phi$, such that for all $t\in \R$
    \begin{equation}\label{eq:local_dispersion}
        \|\e^{\I\sigma\Omega(D)t}\phi\|_{L^\infty}\leq C\langle||N|-|f||t\rangle^{-\frac12}.
    \end{equation}
\end{lemma}
\begin{proof}
    It suffices to consider $|t(|N|-|f|)|>1$, since the remaining case follows from the direct bound 
    \begin{equation}
        \norm{\e^{\I\sigma\Omega(D)t}\phi}_{L^\infty}\leq C\|\hat\phi\|_{L^1_\xi}.
    \end{equation}
    
    Using polar coordinates $\xi=(rOs\theta,r\sin\theta)$, the symbol $\Omega(\theta)=\sqrt{N^2Os^2\theta+f^2\sin^2\theta}$ is independent of $r$, since it is $0$-homogeneous, and therefore the semigroup factorizes as
    \begin{equation} \label{eq:lemma22_step1}
        \e^{\I\sigma\Omega(D)t}\phi(x)=\int_0^{2\pi }\e^{\I\sigma t\Omega(\theta)}K_x(\theta)\dd\theta,
    \end{equation}
    with 
    \begin{equation}\label{eq:def_K}
         \qquad K_x(\theta):=\frac1{4\pi^2}\int_{3/4}^{8/3} \e^{\I r(x_1Os\theta+x_3\sin\theta)}\hat{\phi}(rOs\theta,r\sin\theta)r\dd r.
    \end{equation}
    A direct computation gives
    \begin{equation}\label{eq:derivatives_omega_theta}
        \partial_\theta \Omega =\frac{(f^2-N^2)\sin\theta\cos\theta}{\Omega},\qquad \partial_\theta^2\Omega=-\left(\Omega-\frac{|Nf|}{\Omega}\right)\left(1+\frac{|Nf|}{\Omega^2}\right),
    \end{equation}
    and the identity 
    \begin{equation}\label{eq:identity_dethetaOmega}
        (|N|-|f|)^2-|\partial_\theta\Omega|^2=\left(\Omega-\frac{|Nf|}{\Omega}\right)^2.
    \end{equation}
    To apply the van der Corput lemma to \eqref{eq:lemma22_step1}, we first combine \eqref{eq:derivatives_omega_theta} with \eqref{eq:identity_dethetaOmega} to deduce the pointwise estimate
    \begin{equation}\label{eq:key_ineq_Omega}
        |\de_\theta\Omega|^2+|\de_\theta^2\Omega|^2\geq (|N|-|f|)^2.
    \end{equation}
    Hence, when $|\de_\theta\Omega|^2>\frac{1}{2}\big||N|-|f|\big|^2>0$, the van der Corput lemma, applied with the first derivative, gives on a uniformly bounded number of intervals where $\de_\theta\Omega$ is monotone, gives 
    \begin{equation}\label{eq:outer_Omega}
        \left|\int_{\{|\de_\theta\Omega|>\frac{1}{\sqrt2}\big||N|-|f|\big|\}}\e^{\I\sigma t\Omega(\theta)}K_x(\theta)\dd\theta\right|\lesssim \big| (|N|-|f|)t\big|^{-1} \left(\norm{K_x}_{L^\infty}+\norm{\de_\theta K_x}_{L^1}\right).
    \end{equation}
    On the other hand, if $|\de_\theta\Omega|^2\leq\frac{1}{2}\big||N|-|f|\big|^2$, then \eqref{eq:key_ineq_Omega} implies that $|\de^2_\theta\Omega|^2>\frac{1}{2}\big||N|-|f|\big|^2>0$, and so van der Corput Lemma, applied with the second derivative, yields
    \begin{equation}\label{eq:inner_Omega}
        \left|\int_{\{|\de_\theta\Omega|\leq\frac{1}{\sqrt2}\big||N|-|f|\big|\}}\e^{\I\sigma t\Omega(\theta)}K_x(\theta)\dd\theta\right|\lesssim \big| (|N|-|f|)t\big|^{-\frac12} \left(\norm{K_x}_{L^\infty}+\norm{\de_\theta K_x}_{L^1}\right).
    \end{equation}
    Combining \eqref{eq:inner_Omega} with \eqref{eq:outer_Omega} and using that $\big|(|N|-|f|)t\big|>1$ finally gives 
    \begin{equation}
        \left|\int_0^{2\pi}\e^{\I\sigma t\Omega(\theta)}K_x(\theta)\dd\theta\right|\lesssim\big| (|N|-|f|)t\big|^{-\frac12} \left(\norm{K_x}_{L^\infty}+\norm{\de_\theta K_x}_{L^1}\right).
    \end{equation}
    The lemma then follows once we show that the norms of $K_x$ are bounded uniformly in $x$.
    To this extent, note that $\norm{K_x}_{L^\infty}\leq C_\phi$ follows from \eqref{eq:def_K}, while differentiation in $\theta$ followed by two integrations by parts in $r$, using the compact support of $\hat\phi$, yields \begin{equation*}\begin{aligned} \norm{\de_\theta K_x}_{L^1}&\lesssim_\phi 1+\int_0^{2\pi}\frac{|-x_1\sin\theta+x_3Os\theta|}{1+(x_1Os\theta+x_3\sin\theta)^2}\dd\theta\\&=1+4\arctan|x|\lesssim1.\end{aligned}\end{equation*}
    As a consequence, the localized dispersive estimate \eqref{eq:local_dispersion} holds, which concludes the proof.
\end{proof}
A direct consequence of Lemma~\ref{lemma:local_dispersion} is the following dispersive estimate in homogeneous Besov spaces.
\begin{proposition} \label{prop:general_dispersive_estimate}
Fix $\sigma\in\{\pm\}$ and assume $N^2\neq f^2$. Let $\phi\in \dot B^{s+2}_{1,q}$ for some $q\in[1,\infty]$ and $s\in\R$. 
Then, for all $t\in \R$, it holds
\begin{align}\label{eq:general_dispersive_estimate}
    \|\e^{\I\sigma\Omega(D)t}\phi\|_{\dot B^{s}_{\infty,q}}&\lesssim_{N,f} \langle t\rangle^{-\frac12}\|\phi\|_{\dot{B}^{s+2}_{1,q}}.
\end{align}
\end{proposition}
\begin{proof}
    For $q<\infty$, using the definition of $\dot B^s_{\infty,q}$ (see Appendix~\ref{app:appendixBesov}) and Young's inequality for convolutions, we obtain
\begin{align}
    \|\e^{\I\sigma\Omega(D)t}\phi\|_{\dot B^s_{\infty,q}}^q
    &\lesssim\sum_{j\in\mathbb Z} 2^{jsq}\|\dot\Delta_j \e^{\I\sigma\Omega(D)t}*\tilde{\dot\Delta}_j\phi\|_{L^\infty}^q
    \lesssim \sum_{j\in\mathbb Z}2^{jsq}\|\dot\Delta_{j}\e^{\I\sigma\Omega(D)t}\|_{L^\infty}^q \|\tilde{\dot\Delta}_j\phi\|_{L^1}^q,
\end{align}
where $\tilde{\dot\Delta}$ is a slightly fattened Littlewood--Paley projector. For $q=\infty$, the same argument applies with suprema in place of sums.
The proposition then follows after proving that
\begin{equation}\label{eq:dispersion_claim}
    \|\dot\Delta_{j}\e^{\I\sigma\Omega(D)t}\|_{L^\infty}=2^{2j} \|\dot\Delta_{0}\e^{\I\sigma\Omega(D)t}\|_{L^\infty}\lesssim_{N,f}2^{2j}\langle t\rangle^{-\frac12},
\end{equation}
where $\dot\Delta_{j}\e^{\I\sigma\Omega(D)t}$ denotes (with a slight abuse of notation) the convolution kernel of the operator $\dot\Delta_{j}\e^{\I\sigma\Omega(D)t}$.
To this extent, the first identity in \eqref{eq:dispersion_claim} follows from the $0$-homogeneity of $\Omega$, since
\begin{equation}
    \dot\Delta_{j}\e^{\I\sigma\Omega(D)t}(x) := \frac1{4\pi^2}\int_{\mathbb R^2} \varphi(2^{-j}\xi) \mathrm{e}^{\mathrm{i} x\cdot\xi+\I\sigma t\Omega(\xi)} \,\dd\xi= 2^{2j} \dot\Delta_{0}\e^{\I\sigma\Omega(D)t}(2^jx),
\end{equation}
while Lemma~\ref{lemma:local_dispersion} proves the inequality in \eqref{eq:dispersion_claim}, hence concluding the proof.
\end{proof}
The last proposition in this section is devoted to proving that the dispersive bound \eqref{eq:local_dispersion} is sharp.
\begin{proposition}\label{prop:sharp_dispersion}
Fix $\sigma\in\{\pm\}$ and assume $N^2\neq f^2$. There exists $\psi\in\mathcal{C}^\infty_c(\R^2)$ such that, for all $|t|$ sufficiently large, it holds
\begin{equation}
    \|\e^{\I\sigma\Omega(D) t}\psi\|_{L^\infty} \sim_{N,f} |t|^{-\frac12}.
\end{equation}
\end{proposition}
\begin{proof}
Choose a radial function $\psi\in\cC_c^\infty(\R^2)$ with $\psi(0)=1$. Proposition~\ref{prop:general_dispersive_estimate} gives
\begin{equation}
\|\e^{\I\sigma\Omega(D) t}\psi\|_{L^\infty} \lesssim_{N,f} |t|^{-\frac12}\|\psi\|_{\dot B^2_{1,1}} \lesssim_{N,f,\psi} |t|^{-\frac12}, \qquad |t|\ge1.
\end{equation}
On the other hand, since $\hat\psi$ is radial and $\Omega$ is $0$-homogeneous, polar coordinates $(r,\theta)$ and Fourier inversion at the origin give
\begin{equation}\label{eq:sharp_angular_integral}
(\e^{\I\sigma\Omega(D) t}\psi)(0) =\frac1{2\pi} \int_0^{2\pi} \e^{\I\sigma t\sqrt{N^2Os^2\theta+f^2\sin^2\theta}} \dd\theta.
\end{equation}
The angular phase has precisely four stationary points on the circle, at $0,\frac\pi2,\pi,\frac{3\pi}{2}$, at which its second derivative takes the values
\begin{equation}
    \left.\frac{\dd^2}{\dd\theta^2}\Omega(\theta)\right|_{\theta=0,\pi}=\frac{f^2-N^2}{|N|},\qquad
    \left.\frac{\dd^2}{\dd\theta^2}\Omega(\theta)\right|_{\theta=\frac\pi2,\frac{3\pi}2}=\frac{N^2-f^2}{|f|}.
\end{equation}
Thus, all four stationary points are non-degenerate, and we can apply the one-dimensional stationary phase formula to \eqref{eq:sharp_angular_integral} to get
\begin{align}
\begin{split}
    (\e^{\I\sigma\Omega(D) t}\psi)(0) &= \sqrt{\frac{2}{\pi|N^2-f^2|}}|t|^{-\frac12}\\ &\quad\times\Bigl( \sqrt{|N|} \e^{\I\sigma t|N| -\frac{\I\pi}{4}\operatorname{sgn}(\sigma t(N^2-f^2))} +\sqrt{|f|}\e^{\I\sigma t|f|+\frac{\I\pi}{4}\operatorname{sgn}(\sigma t(N^2-f^2))}\Bigr)\\
    &\qquad+O_{N,f}(|t|^{-\frac32}).
\end{split}
\end{align}
Here the phase and amplitude are independent of $t$, and the second derivatives at the stationary points are fixed nonzero numbers, hence the implicit constant in the remainder depends only on $N$ and $f$.
By the reverse triangle inequality, there exists a constant $C_{N,f}>0$ depending on $N,f$ such that
\begin{equation}\label{eq:pointwise_sharp_bound}
    \|\e^{\I\sigma\Omega(D) t}\psi\|_{L^\infty} \ge |(\e^{\I\sigma\Omega(D) t}\psi)(0)| \ge \sqrt{\frac{2}{\pi|N^2-f^2|}}\, \left|\sqrt{|N|}-\sqrt{|f|}\right|\,|t|^{-\frac12} - C_{N,f}|t|^{-\frac32}.
\end{equation}
Note that, since $|N|\ne|f|$, the coefficient of $|t|^{-\frac12}$ in \eqref{eq:pointwise_sharp_bound} is strictly positive. The remainder can therefore be absorbed for all sufficiently large $|t|$, giving
\begin{equation}
    \|\e^{\I\sigma\Omega(D) t}\psi\|_{L^\infty}\gtrsim_{N,f} |t|^{-\frac12},
\end{equation}
which completes the proof.
\end{proof}
Analogous sharpness statements for $f=0$ can be found in \cite{ElgWid2015}.

From Proposition~\ref{prop:sharp_dispersion}, it follows that the decay rate in \eqref{eq:general_dispersive_estimate} is sharp.
Indeed, since $\psi(0)=1$, the Littlewood--Paley decomposition ensures that there exists $j_0\in\ZZ$ with $(\dot\Delta_{j_0}\psi)(0)\neq0$.
Using radial dyadic cutoffs, the same computation as in \eqref{eq:sharp_angular_integral} gives
\begin{equation}
(\dot\Delta_{j_0}\e^{\I\sigma\Omega(D)t}\psi)(0)=\frac{(\dot\Delta_{j_0}\psi)(0)}{2\pi}\int_0^{2\pi}\e^{\I\sigma t\sqrt{N^2Os^2\theta+f^2\sin^2\theta}}\dd\theta.
\end{equation}
The lower bound for this angular integral established in \eqref{eq:pointwise_sharp_bound} yields, for every $s\in\R$, $1\le q\le\infty$, and sufficiently large $|t|$,
\begin{equation}
\|\e^{\I\sigma\Omega(D)t}\psi\|_{\dot B^s_{\infty,q}}\ge 2^{j_0s} |(\dot\Delta_{j_0}\e^{\I\sigma\Omega(D)t}\psi)(0)|\gtrsim_{N,f} |t|^{-\frac12}.
\end{equation}
This lower bound, together with the upper bound \eqref{eq:general_dispersive_estimate} applied to $\dot\Delta_{j_0}\psi\in\dot B^{s+2}_{1,q}$ (whose Fourier support lies in a fixed annulus) proves the claimed sharpness.

\section{Nonlinear dynamics} \label{sec:nonlinear_dynamics}

In this section, we study the nonlinear dynamics of \eqref{eq:2D3C}.
First, in Section~\ref{subsec:blow_up}, we derive the blow-up criterion in Sobolev norms via energy estimates.
This will later be employed to propagate the bootstrap assumption, see Proposition~\ref{prop:bootstrap} in Section~\ref{subsec:bootstrap}.
This bootstrap argument is the main tool in the proofs of Theorems~\ref{thm:43} and~\ref{thm:2}.
Finally, we devote Section~\ref{subsec:PV} to the study of the potential vorticity $q$, in particular to the propagation of smallness of its Sobolev norms.

\subsection{Blow-up criterion} \label{subsec:blow_up}
Fix $s>4$.
For the system \eqref{eq:2D3C}, we then define the energy $E_s$ by setting
\begin{equation}\label{eq:energy_defin}
    E_s^2:=\norm{u}^2_{H^s}+N^{-2}\norm{b}^2_{H^s}=\norm{U}^2_{H^s}+f^{-2}\norm{a}^2_{H^s}+N^{-2}\norm{b}^2_{H^s},
\end{equation}
where we recall that $U=(u_1,u_3)$ and $a=fu_2$.
The following blow-up criterion holds.
\begin{proposition}\label{prop:blow_up_criterion}
    Let $N,f$ be as in Theorem~\ref{thm:43}. Then there exists a constant $c>0$, depending only on $s$, such that, for all $T\geq 0$,
    \begin{equation}\label{eq:blow-up_Uab}
        E_s^2(T)\leq E_s^2(0)\e^{2c\int_0^TD(t)\dd t},\quad D(t):=\norm{\nabla U(t)}_{L^\infty}+|f|^{-1}\norm{\nabla a(t)}_{L^\infty}+|N|^{-1}\norm{\nabla b(t)}_{L^\infty}\hspace{-2pt}.
    \end{equation}
\end{proposition}
\begin{proof}
    We apply $\cJ_s:=(I-\Delta)^\frac s2$ to the system \eqref{eq:2D3C}, test the momentum and buoyancy equations with $\cJ_su$ and $N^{-2}\cJ_sb$, respectively, and sum the results. Using $\nabla\cdot u=0$, we obtain
    \begin{equation}
        \frac 1 2 \ddt E_s^2(t)=\frac{1}{N^2}\int_{\R^2}\cJ_sb[u\cdot\nabla,\cJ_s]b\dd x+\int_{\R^2}\cJ_su\cdot[u\cdot\nabla,\cJ_s]u\dd x.
    \end{equation}
    A standard Kato--Ponce argument combined with the Cauchy--Schwarz and Young inequalities leads to
    \begin{align}
        \frac12\ddt E_s^2(t)&\lesssim\|\cJ_sb\|_{L^2}\|[u\cdot\nabla,\cJ_s]b\|_{L^2}+\|\cJ_su\|_{L^2}\|[u\cdot\nabla,\cJ_s]u\|_{L^2}\\
        &\lesssim N^{-2}\|b\|_{H^s}^2\|\nabla u\|_{L^\infty}+N^{-2}\|b\|_{H^s}\|u\|_{H^s}\|\nabla b\|_{L^\infty}+\|u\|_{H^s}^2\|\nabla u\|_{L^\infty}\\
        &\lesssim E_s^2(t)\big(\|\nabla u\|_{L^\infty}+|N|^{-1}\|\nabla b\|_{L^\infty}\big).
    \end{align}
    Since $\|\nabla u\|_{L^\infty}\le\|\nabla U\|_{L^\infty}+|f|^{-1}\|\nabla a\|_{L^\infty}$, the bound \eqref{eq:blow-up_Uab} follows by Gr\"onwall lemma.
\end{proof}

\subsection{Bootstrap argument} \label{subsec:bootstrap}
Here, we set up the bootstrap argument that will allow us to prove our main results.
More specifically, Theorems~\ref{thm:43} and~\ref{thm:2} follow from the following proposition, whose proof relies on the blow-up criterion of Proposition~\ref{prop:blow_up_criterion}.
\begin{proposition}\label{prop:bootstrap}
    Under the assumptions of Theorem~\ref{thm:43} (resp. Theorem~\ref{thm:2}), let $T\leq c_0 \Lambda_{N,f} \eps^{-\frac43}$ (resp. $T\leq \tilde c_0\tilde \Lambda_{N,f} \eps^{-2}$).
    If
    \begin{equation}\label{eq:bootstrap_assumption}
        E_s^2(t)\leq 2\varepsilon^2,\quad\forall t\in[0,T],
    \end{equation}
    then the same estimate holds on $[0,T]$ with $2$ replaced by $\frac32$.
\end{proposition}
The proof of Proposition~\ref{prop:bootstrap} is deferred to Section~\ref{sec:proofof43} (resp. Section~\ref{sec:proofof2}).
Then, the proof of Theorem~\ref{thm:43} (resp. Theorem~\ref{thm:2}) is straightforward.
\begin{proof}[Proof of Theorems~\ref{thm:43} and~\ref{thm:2}]
    Given Proposition~\ref{prop:bootstrap}, Theorem~\ref{thm:43} (resp. Theorem~\ref{thm:2}) follows via a standard bootstrap argument.
\end{proof}
Proposition~\ref{prop:blow_up_criterion} and Proposition~\ref{prop:bootstrap} are the main ingredients in extending the lifespan of solutions, and they point directly towards a more careful analysis of $D(t)$ ( recall \eqref{eq:blow-up_Uab}).
Indeed, in view of Proposition~\ref{prop:blow_up_criterion}, it is clear that Proposition~\ref{prop:bootstrap} follows upon proving
\begin{equation} \label{eq:proof_of_bootstrap_bu}
    \int_0^{T} D(t)\dd t\le\frac{1}{2c}\log\left(\frac32\right),\quad\forall T\in[0,T_\eps],
\end{equation}
where $c>0$ is the constant appearing in the blow-up criterion \eqref{eq:blow-up_Uab}.
\begin{remark}
    Note that, if $s>2=\frac{d}2+1$ (here $d=2$), the standard Sobolev embedding implies $D(t)\lesssim E_s(t)$, leading to the well-know existence time $T_\eps\gtrsim\eps^{-1}$ for initial data of size $\eps$.
\end{remark}
To study the time integral of $D(t)$, we use the following proposition, which allows us to work with the variables $(q,z_+,z_-)$, exploiting both the well-preparedness assumption for $q$ and the dispersive behavior encoded by $(z_+,z_-)$.
\begin{proposition}\label{prop:D_bound_general}
    The quantity $D(t)$ satisfies
    \begin{equation}\label{eq:D_bound_general}
        D(t)\lesssim_{N,f} \|z_-(t)\|_{\dot B^1_{\infty,1}}+\|z_+(t)\|_{\dot B^1_{\infty,1}}+ \norm{q(t)}_{\dot B^0_{\infty,1}}+\norm{J(a,b)(t)}_{\dot B^0_{\infty,1}}.
    \end{equation} 
\end{proposition}
\begin{proof}
    After inverting the relations \eqref{eq:potential_vorticity} and \eqref{eq:eigenfunctions_real}, we can express $(U,a,b)$ in terms of $(q,z_+,z_-)$ as 
    \begin{equation}\label{eq:D_bound_change_variables}
    \begin{cases}
    U=\begin{bmatrix}
        -\de_3|\nabla|^{-1}(z_++z_-)\\
        \de_1|\nabla|^{-1}(z_++z_-)
    \end{bmatrix}, \\
    a=-\Omega^{-2}(D)\de_1|\nabla|^{-2}(q-J(a,b))-\I f^2\Omega^{-1}(D)\de_3|\nabla|^{-1}(z_+-z_-),\\
    b=-\Omega^{-2}(D)\de_3|\nabla|^{-2}(q-J(a,b))+\I N^2\Omega^{-1}(D)\de_1|\nabla|^{-1}(z_+-z_-).
    \end{cases}
    \end{equation}
    Taking the gradient of \eqref{eq:D_bound_change_variables} and grouping in terms of $z_\pm$, $q$, and $J(a,b)$ shows that all the Fourier multipliers involved are of order $0$ or $1$. 
    More precisely, $q$ and $J(a,b)$ carry the $0$-homogeneous ones, while the $1$-homogeneous ones are applied to $z_\pm$.
    We can then combine the triangle inequality with the embeddings $\dot B^{0}_{\infty,1}\hookrightarrow L^\infty$ and $\dot B^{1}_{\infty,1}\hookrightarrow \dot W^{1,\infty}$ to deduce \eqref{eq:D_bound_general}.
\end{proof}
In view of Proposition~\ref{prop:D_bound_general}, we will divide the proof of Proposition~\ref{prop:bootstrap} into two parts.
Since the initial datum is well-prepared, $q\equiv 0$ and the corresponding term on the right-hand side of \eqref{eq:D_bound_general} vanishes. 
Hence, we only have to control the oscillating variables $z_\pm$ and the nonlinear term $J(a,b)$.
While the latter can be handled in the same way in the $\eps^{-\frac43}$ and in the $\eps^{-2}$ cases (cf. Lemma~\ref{lemma:control_of_J}), the former requires different treatments, given in Sections~\ref{sec:proofof43} and~\ref{sec:proofof2}, respectively.

As pointed out in Remark~\ref{rmk:gamma_well_prepared}, the assumption on the initial potential vorticity can be relaxed.
In that case, the corresponding term on the right-hand side of \eqref{eq:D_bound_general} can be bounded by exploiting the propagation of smallness in Sobolev norms, which is the content of the next subsection.

\subsection{On the Sobolev norms of the potential vorticity} \label{subsec:PV}

As discussed in Section~\ref{subsec:about_the_theorems}, the potential vorticity
\begin{equation}
    q=N^2\de_1a+f^2\de_3b+J(a,b)
\end{equation}
is a key transported quantity for the system \eqref{eq:2D3C}, and the well-preparedness assumption \eqref{ass:wp} forces $q\equiv0$ at all times.
Although the conservation of $L^p$ norms is standard for transported quantities, the propagation of smallness in terms of Sobolev norms requires an additional argument.
Such a result would allow one to relax the well-preparedness assumption \eqref{ass:wp} to \eqref{ass:gamma_wp} for some $\gamma$ large enough.

As pointed out in \cite{XieBuhler2019,ShavitBuhlerShatah2026}, beside the potential vorticity, two further scalar quantities are transported by $U$.
Together, they define a change of coordinates $\Psi:\R^2\to\R^2$, given by
\begin{equation}
    \Psi(t,x):=\begin{bmatrix}
        x_1+f^{-2}a(t,x)\\
        x_3+N^{-2}b(t,x)
    \end{bmatrix},
\end{equation}
which satisfies the transport equation
\begin{equation}
    \left(\de_t+u_1\de_1+u_3\de_3\right)\Psi=0.
\end{equation}
Using these coordinates, we are prove the following result for any quantity $\cQ$ transported by $U$.
In our applications, $\cQ$ is the potential vorticity $q$.
\begin{proposition} \label{proposition:q_propagation}
    Fix $s>2$ and let $\cQ(t,x_1,x_3)$ be the solution of
    \begin{equation}
        (\de_t+u_1\de_1+u_3\de_3)\cQ=0,\quad\cQ|_{t=0}=\cQ_0\in H^s.
    \end{equation}
    Assume that there exists $T>0$ such that
    \begin{equation}
        \norm{(a,b)(t,\cdot)}_{H^s}\leq \eps_*\quad\forall t\in[0,T],
        \label{eq:assumption_invertibility}
    \end{equation}
    for some $\eps_*>0$ small enough (depending on $N,f$).
    Then, there exists a constant $C=C(\eps_*)>0$ such that, for all $ t\in[0,T]$,
    \begin{equation} \label{eq:bounds_Hs_Q}
        C^{-1}\norm{\cQ_0}_{H^s}\leq \norm{\cQ(t,\cdot)}_{H^s}\leq C\norm{\cQ_0}_{H^s}.
    \end{equation}
\end{proposition}
\begin{proof}
    Since both $\cQ$ and $\Psi$ are transported by the same velocity field $U$, denoting by $X_t:\R^2\to\R^2$ the associated flow map at time $t\in[0,T]$, we have
    \begin{equation} \label{eq:transported_quantities}
        \cQ(t,x)=\cQ_0(X_t^{-1}(x)),\quad\Psi(t,x)=\Psi_0(X_t^{-1}(x)),
    \end{equation}
    where $\Psi_0:=\Psi(0,\cdot)$.
    Note that, in view of assumption \eqref{eq:assumption_invertibility}, the two elements on the main diagonal of $\nabla\Psi(t,x)$ are nonzero. Indeed, one has
    \begin{equation}
        |(\nabla\Psi(t,x))_{1,1}|=|1+f^{-2}\de_1a(t,x)|\ge1-f^{-2}\norm{a}_{H^s}\ge1-f^{-2}\varepsilon_*>0,
    \end{equation}
    and, analogously,
    \begin{equation}
        |(\nabla\Psi(t,x))_{2,2}|=|1+N^{-2}\de_3b(t,x)|\ge1-N^{-2}\norm{b}_{H^s}\ge1-N^{-2}\varepsilon_*>0.
    \end{equation}
    In particular, we can invert the change of coordinates $\Psi$ at $t=0$.
    Combining this fact with \eqref{eq:transported_quantities}, we have
    \begin{equation} \label{eq:cQ_transported}
        \cQ(t,x)=\cQ_0(\Psi_0^{-1}(\Psi(t,x))).
    \end{equation}
    Appealing to \cite[Lemmas B.5-B.6]{IncKapTop2013}, assumption \eqref{eq:assumption_invertibility} also implies that
    \begin{equation}
        \norm{\Psi_0^{-1}\circ\Psi(t,\cdot)-\operatorname{Id}}_{H^s}\lesssim\varepsilon_*,
    \end{equation}
    therefore from \eqref{eq:cQ_transported} and \cite[Lemma B.5]{IncKapTop2013} again we get
    \begin{equation}
        \norm{\cQ(t,\cdot)}_{H^s}\lesssim\norm{\cQ_0}_{H^s}.
    \end{equation}
    Arguing analogously with
    \begin{equation}
        \cQ_0(x)=\cQ(t,X_t(x)),\quad\Psi_0(x)=\Psi(t,X_t(x)),
    \end{equation}
    in place of \eqref{eq:transported_quantities}, we also obtain the reverse inequality in \eqref{eq:bounds_Hs_Q}, hence concluding the proof.
\end{proof}

\section{Bootstrap in the generic regime} \label{sec:proofof43}

This section is devoted to proving Proposition~\ref{prop:bootstrap} with $T_\eps\sim_{N,f}\varepsilon^{-\frac43}$.
Bearing in mind Proposition~\ref{prop:D_bound_general}, we first turn our attention to the nonlinear term $J(a,b)$.
\begin{lemma} \label{lemma:control_of_J}
    Under the assumptions of Proposition~\ref{prop:bootstrap}, for $\eps>0$ sufficiently small, it holds
    \begin{equation}
        \int_0^{T_\eps}\norm{J(a,b)(t)}_{\dot B^0_{\infty,1}}\dd t\leq \frac{1}{4c}\log\left(\frac32\right).
    \end{equation}
\end{lemma}
\begin{proof}
    Let $s>2$. Using the embeddings $H^{s-1}\hookrightarrow\dot B^1_{2,1}\hookrightarrow\dot B^0_{\infty,1}$, the fact that $H^{s-1}$ is an algebra, and the bootstrap assumption \eqref{eq:bootstrap_assumption}, we obtain
    \begin{equation}
        \norm{J(a,b)(t)}_{\dot B^0_{\infty,1}}\lesssim \norm{J(a,b)(t)}_{\dot B^1_{2,1}} \lesssim \norm{a}_{H^s}\norm{b}_{H^s}\lesssim \eps^2.
    \end{equation}
    Integrating in time and choosing $\eps$ sufficiently small then concludes the proof.
\end{proof}
We now aim to control the oscillating variables $z_\pm$ in the $L^1_t(\dot B^1_{\infty,1})_x$-norm.
The estimate that we are going to derive relies on the dispersive estimate in Proposition~\ref{prop:general_dispersive_estimate}, and is contained in the following lemma.
\begin{lemma} \label{lemma:control_of_z_pm}
    Under the assumptions of Proposition~\ref{prop:bootstrap}, for $\eps>0$ sufficiently small it holds
    \begin{equation}
        \int_0^{T_\eps}\left(\|z_-(t)\|_{\dot B^1_{\infty,1}}+\|z_+(t)\|_{\dot B^1_{\infty,1}}\right)\dd t\leq \frac{1}{4c}\log\left(\frac32\right).
    \end{equation}
\end{lemma}
\begin{proof}
    Let $\sigma\in\{\pm\}$ and recall that, after the change of variables $(U,a,b)\mapsto(q,z_+,z_-)$, the oscillating variables $z_\sigma$ satisfy
    \begin{equation}
        \de_t z_\sigma=\I\sigma\Omega(D)z_\sigma+\cN\cL_\sigma,
        \label{eq:evolution_zeta_sigma}
    \end{equation}
    where
    \begin{equation}
        \mathcal{NL}_\sigma=-\frac{1}{2}\nabla\cdot\left(|\nabla|^{-1}U\omega+ \mathrm{i}\sigma\Omega^{-1}(D)|\nabla|^{-1}(\de_1(Ub)-\de_3(Ua))\right).
        \label{eq:nonlinear_term_sigma}
    \end{equation}
    Then, from the Duhamel formula
    \begin{equation}
        z_\sigma(t)=\e^{\I\sigma\Omega(D)t}z_\sigma(0)+\int_0^t\e^{\I\sigma\Omega(D)(t-\tau)}\cN\cL_\sigma(\tau)\dd \tau,
    \end{equation}
    we obtain, using the dispersive estimate \eqref{eq:general_dispersive_estimate},
    \begin{equation} \label{eq:step_43_1}
        \norm{z_\sigma(t)}_{\dot B^1_{\infty,1}}\lesssim_{N,f} \langle  t\rangle^{-\frac12}\norm{z_\sigma(0)}_{\dot B^3_{1,1}}+\int_0^t\langle t-\tau\rangle^{-\frac12}\norm{\cN\cL_\sigma(\tau)}_{\dot B^3_{1,1}}\dd \tau.
    \end{equation}
    To bound the nonlinear terms, we use that $s>4$ and we combine Proposition~\ref{prop:product_estim_Besov} with the bootstrap assumption \eqref{eq:bootstrap_assumption} to get
    \begin{align}
    \begin{aligned}\label{eq:control_quadratic_nonlinearity}
        \norm{\cN\cL_\sigma}_{\dot B^3_{1,1}}&\lesssim _{N,f}\norm{U\omega}_{\dot B^3_{1,1}}+\norm{Ua}_{\dot B^4_{1,1}}+\norm{Ub}_{\dot B^4_{1,1}}\\
        &\lesssim_{N,f}\norm{U}_{H^s}(\norm{\omega}_{H^{s-1}}+\norm{a}_{H^s}+\norm{b}_{H^s})\\
        &\lesssim_{N,f}\eps^2.
    \end{aligned}
    \end{align}
    Since $z_\sigma$ has the same regularity level as $U,a,$ and $b$, we can use \eqref{eq:thm1_assumption_ic_0} and Corollary~\ref{coroll:0homo_multi} to bound the initial datum on the right-hand side of \eqref{eq:step_43_1}.\footnote{Up to zero order multipliers which are bounded on homogeneous Besov spaces, see Corollary~\ref{coroll:0homo_multi} and Proposition~\ref{prop:omega-1_bound}.} 
    Combining this with the estimate above, we obtain
    \begin{equation} \label{eq:estimate_to_be_improved_z_sigma}
        \int_0^{T_\eps}\norm{z_\sigma(t)}_{\dot B^1_{\infty,1}}\dd t\lesssim_{N,f}\int_0^{T_\eps}\langle  t\rangle^{-\frac12}\eps\dd t+\int_0^{T_\eps}\int_0^t\langle t-\tau\rangle^{-\frac12}\eps^2\dd \tau\dd t\lesssim_{N,f} T_\eps^\frac12 \eps + T_\eps^{\frac32}\eps^2.
    \end{equation}
    Since $T_\eps\lesssim_{N,f}\eps^{-\frac43}$, choosing $\eps$  and $c_0$ small enough concludes the proof.
\end{proof}
We are now in a position to prove Proposition~\ref{prop:bootstrap}.
\begin{proof}[Proof of Proposition~\ref{prop:bootstrap}]
    Combining Lemmas~\ref{lemma:control_of_J} and~\ref{lemma:control_of_z_pm}, and using that $q\equiv0$, since the initial datum for \eqref{eq:2D3C} is well-prepared, \eqref{eq:proof_of_bootstrap_bu} follows.
    An application of the blow-up criterion in Proposition~\ref{prop:blow_up_criterion} then concludes the proof.
\end{proof}

\section{Bootstrap in the non-resonant regime} \label{sec:proofof2}

In this section we prove Proposition~\ref{prop:bootstrap} with $T_\eps\sim_{N,f}\varepsilon^{-2}$.
To do so, we estimate the oscillating variables $z_\pm$ more carefully, since Lemma~\ref{lemma:control_of_z_pm} is the bottleneck that restricts the argument of Section~\ref{sec:proofof43} to the $\varepsilon^{-\frac43}$ timescale.
Note that the nonlinear term $J(a,b)$ can be estimated analogously to Lemma~\ref{lemma:control_of_J}.
To this end, we use a normal form method \cite{Ger10}, which applies to the specific regime
\begin{equation}\label{eq:normal_form_regime}
    2\min\{|N|,|f|\}>\max\{|N|,|f|\},
\end{equation}
and allows us to improve the stability timescale for solutions to \eqref{eq:2D3C} from $T_\eps\sim_{N,f}\eps^{-\frac43}$ to $T_\eps\sim_{N,f}\eps^{-2}$.
This hinges upon the following observation.
A closer inspection of the dispersion relation defined in \eqref{eq:dispersion_relation_fourier}, shows that, for fixed $N,f\neq 0$  with $|N|\neq |f|$, the symbol $\Omega$ satisfies
\begin{equation}
    0<\min\{|N|,|f|\}\leq\Omega(\xi)\leq \max\{|N|,|f|\},\qquad \forall \xi \in \R^2.
\end{equation}
Throughout the rest of this section, we assume \eqref{eq:normal_form_regime} unless stated otherwise.

We now state the result that, in the normal form argument, will allow us to integrate by parts in time and obtain a cubic nonlinearity from the quadratic one.
Setting $\delta_*:=2\min\{|N|,|f|\}-\max\{|N|,|f|\}$, in the regime \eqref{eq:normal_form_regime} the phase
\begin{equation}
    \Phi_{\sigma,\nu\mu}(\xi,\eta)=-\sigma \Omega(\xi)+\nu\Omega(\xi-\eta)+\mu\Omega(\eta), \qquad \sigma,\nu,\mu\in\{\pm\},
    \label{eq:Phi_hat}
\end{equation}
is non-resonant, i.e. bounded away from zero uniformly in $(\xi,\eta)$.
\begin{lemma} \label{lemma:alg_Phi}
    Let $N,f\ne0$ satisfy $N^2\neq f^2$ and \eqref{eq:normal_form_regime}.
    Then, for all $\xi,\eta\in\R^2$ and all $\sigma,\nu,\mu\in\{\pm\}$ it holds 
    \begin{equation}
        |\Phi_{\sigma,\nu\mu}(\xi,\eta)|\geq\delta_*.
        \label{eq:algebra_lemma_th}
    \end{equation}
\end{lemma}
\begin{proof}
    For any given triple $(-\sigma,\nu,\mu)\in\{\pm\}^3$, let $\kappa\in\{\pm\}$ denote the most frequent sign of the three, equivalently characterized as
    \begin{equation}
        \kappa=\operatorname{sgn}(-\sigma+\nu+\mu).
        \label{eq:def_of_kappa}
    \end{equation}
    If $|-\sigma+\nu+\mu|=3$ (i.e. all three signs coincide), we have
    \begin{equation}
        \kappa\Phi_{\sigma,\nu\mu}(\xi,\eta)= \Omega(\xi)+\Omega(\xi-\eta)+\Omega(\eta)\ge3\min\{|N|,|f|\}\ge\delta_*.
    \end{equation}
    On the other hand, if $|-\sigma+\nu+\mu|=1$ (i.e. exactly one sign differs from the other two), we have three possibilities:
    \begin{equation}
        \kappa\Phi_{\sigma,\nu\mu}(\xi,\eta)=\begin{cases}
            -\Omega(\xi)+\Omega(\xi-\eta)+\Omega(\eta),\quad&\text{if }\;\nu=\mu,\\
            \Omega(\xi)-\Omega(\xi-\eta)+\Omega(\eta),\quad&\text{if }\;-\sigma=\mu,\\
            \Omega(\xi)+\Omega(\xi-\eta)-\Omega(\eta),\quad&\text{if }\;-\sigma=\nu.
        \end{cases}
    \end{equation}
    In each of these three cases, the bound $\kappa\Phi_{\sigma,\nu\mu}(\xi,\eta)\ge\delta_*>0$ holds.
    The proof is then complete since
    \begin{equation}
        |\Phi_{\sigma,\nu\mu}(\xi,\eta)|=|\kappa\Phi_{\sigma,\nu\mu}(\xi,\eta)|=\kappa\Phi_{\sigma,\nu\mu}(\xi,\eta)\ge\delta_*.
    \end{equation}
\end{proof}
Note that, in the regime \eqref{eq:normal_form_regime}, we have $\kappa=\sign{\Phi_{\sigma,\nu\mu}}$.

\subsection{Normal form}
We now introduce the change of variables that allows us to apply the normal form approach.
We begin by observing that, for $\sigma\in\{\pm\}$, there exist finitely many Fourier multipliers homogeneous of order zero $\{\cA^{0,j}_{\sigma,\nu\mu},\cA^{1,j}_{\sigma,\nu\mu}\}_{\nu,\mu\in\{\pm\},j\in\{1,\dots,\ell\}}$, with $\ell\in\NN$ fixed, and Fourier multipliers homogeneous of order one $\{\cA^{2,j}_{\sigma,\nu\mu}\}_{\nu,\mu\in\{\pm\},j\in\{1,\dots,\ell\}}$ such that the nonlinear term $\mathcal{NL}_\sigma$ can be expressed in $(q,z_+,z_-)$ variables as
\begin{equation}
    \mathcal{NL}_\sigma=\sum_{j=1}^\ell\sum_{\nu,\mu\in\{\pm\}} \cA^{0,j}_{\sigma,\nu\mu}(D)\left(\cA^{1,j}_{\sigma,\nu\mu}(D)z_\nu\cA^{2,j}_{\sigma,\nu\mu}(D)z_\mu\right)+\mathcal E_\sigma.
\end{equation}
Indeed, combining the explicit form in \eqref{eq:nonlinear_term_sigma} with the change of variables \eqref{eq:normal_form_change_of_var} shows that all multipliers $\cA$ are a finite product of operators in the family $\{\Omega,\Omega^{-1}, \de_1,\de_3, \frac{\de_1}{|\nabla|},\frac{\de_3}{|\nabla|}\}$, see for instance the explicit computations in \cite{ShavitBuhlerShatah2026}.
In the above, $\mathcal{E}_\sigma$ collects all the nonlinear interactions that do not occur between the wave modes $z_\pm$. Recalling that in our setting the potential vorticity $q\equiv0$, it reads
\begin{equation} \label{eq:def_E}
    \cE_\sigma=\frac{\mathrm{i}\sigma }{2}\nabla\cdot\Omega^{-1}(D)|\nabla|^{-1}\bigg(\de_3(U\de_1\Omega^{-2}(D)|\nabla|^{-2}J(a,b))-\de_1(U\de_3\Omega^{-2}(D)|\nabla|^{-2}J(a,b))\bigg).
\end{equation}
For any $\sigma,\nu,\mu\in\{\pm\}$, let us define the following pseudo-product operators
\begin{equation} \label{eq:def_B}
    \cB_{\sigma,\nu\mu}(g,h)(x):=\frac{1}{4\pi^2}\cF^{-1}\left(\int_{\R^2}\frac{ m_{\sigma,\nu\mu}(\xi,\eta)}{\I\Phi_{\sigma,\nu\mu}(\xi,\eta)}\hat g(\xi-\eta)\hat h(\eta)\dd \eta \right)(x),\quad g,h\in\cC^\infty_c(\R^2),
\end{equation}
where $\Phi_{\sigma,\nu\mu}$ is given in \eqref{eq:Phi_hat} and the symbols $m_{\sigma,\nu\mu}$ are
\begin{equation} \label{eq:definition_m_symbols}
    m_{\sigma,\nu\mu}(\xi,\eta)=\sum_{j=1}^\ell \cA^{0,j}_{\sigma,\nu\mu}(\xi)\cA^{1,j}_{\sigma,\nu\mu}(\xi-\eta)\cA^{2,j}_{\sigma,\nu\mu}(\eta),\quad\sigma,\nu,\mu\in\{\pm\}.
\end{equation}
With the above notation, the normal form transformation yields the following lemma.
\begin{lemma} \label{lemma:expression_for_z_sigma}
    Fix $\sigma\in\{\pm\}$.
    The diagonalized variable $z_\sigma$ defined in \eqref{eq:eigenfunctions_real} can be expressed as
    \begin{align}
        z_\sigma(t)&=\e^{\I\sigma\Omega(D)t}\bigg[z_\sigma(0)-\sum_{\nu,\mu\in\{\pm\}}\cB_{\sigma,\nu\mu}(z_\nu,z_\mu)(0)\bigg]+\sum_{\nu,\mu\in\{\pm\}}\cB_{\sigma,\nu\mu}(z_\nu,z_\mu)(t)\\
        &\quad - \sum_{\nu,\mu\in\{\pm\}}\int_0^t\e^{\I\sigma\Omega(D)(t-\tau)}\bigg[\cB_{\sigma,\nu\mu}(\mathcal{NL}_\nu,z_\mu)(\tau)+\cB_{\sigma,\nu\mu}(z_\nu,\mathcal{NL}_\mu)(\tau)-\frac14\cE_\sigma(\tau)\bigg]\dd\tau,
\end{align}
where $\cB_{\sigma,\nu\mu}$, $\cN\cL_\sigma$, and $\cE_\sigma$ are defined respectively in \eqref{eq:def_B}, \eqref{eq:nonlinear_term_sigma}, and \eqref{eq:def_E}.
\end{lemma}
\begin{proof}
    Fix $\sigma\in\{\pm\}$ and consider the evolution equation \eqref{eq:evolution_zeta_sigma} for $z_\sigma$. 
    The profile
    \begin{equation}
        y_\sigma(t):=\e^{-\I\sigma\Omega(D)t}z_\sigma(t)
    \end{equation}
    satisfies
    \begin{equation}
        \de_t y_\sigma(t)=\e^{-\I\sigma\Omega(D)t}\mathcal{NL}_\sigma.
        \label{eq:normal_form_change_of_var}
    \end{equation}
    Taking the Fourier transform of \eqref{eq:normal_form_change_of_var} and expressing the nonlinear term $\cN\cL_\sigma$ in terms of the profiles $\hat y_\nu,\hat y_\mu$, we obtain
    \begin{equation}
        \de_t\hat y_\sigma(t,\xi)=\frac{1}{4\pi^2}\sum_{\nu,\mu\in\{\pm\}}\int_{\R^2}\e^{\I\Phi_{\sigma,\nu\mu}(\xi,\eta)t}m_{\sigma,\nu\mu}(\xi,\eta)\hat y_\nu(t,\xi-\eta)\hat y_\mu(t,\eta)\dd \eta+\e^{-\I\sigma\Omega(\xi)t}\hat{\mathcal{E}}_\sigma(t,\xi),
        \label{eq:normal_form_step0}
    \end{equation}
    In light of Lemma~\ref{lemma:alg_Phi}, we can then use the identity 
    \begin{equation}
        \e^{\I\Phi_{\sigma,\nu\mu}(\xi,\eta) \tau}=\frac1{\I\Phi_{\sigma,\nu\mu}(\xi,\eta)}\frac{\dd}{\dd\tau}\e^{\I\Phi_{\sigma,\nu\mu}(\xi,\eta)\tau}
    \end{equation}
    to obtain from \eqref{eq:normal_form_step0} that
    \begin{align}
        \hat y_\sigma(t,\xi)&=\hat y_\sigma(0,\xi)+\int_0^t\e^{-\I\sigma\Omega(\xi)\tau}\hat{\mathcal{E}}_\sigma(\tau,\xi)\dd\tau\\
        &\quad+\frac{1}{4\pi^2}\sum_{\nu,\mu\in\{\pm\}}\int_0^t\int_{\R^2}\frac1{\I\Phi_{\sigma,\nu\mu}(\xi,\eta)}\frac{\dd}{\dd\tau}\e^{\I\Phi_{\sigma,\nu\mu}(\xi,\eta)\tau} m_{\sigma,\nu\mu}(\xi,\eta)\hat y_\nu(\tau,\xi-\eta)\hat y_\mu(\tau,\eta)\dd\eta\dd\tau.
    \end{align}
    Integrating the last term on the right-hand-side by parts in time, we get
    \begin{align}
    \begin{split}
        \hat y_\sigma(t,\xi)&=\hat y_\sigma(0,\xi)+\int_0^t\e^{-\I\sigma\Omega(\xi)\tau}\hat{\mathcal{E}}_\sigma(\tau,\xi)\dd \tau\\
        &\quad+\frac{1}{4\pi^2}\sum_{\nu,\mu\in\{\pm\}}\left[\int_{\R^2}\frac1{\I\Phi_{\sigma,\nu\mu}(\xi,\eta)}\e^{\I\Phi_{\sigma,\nu\mu}(\xi,\eta) \tau} m_{\sigma,\nu\mu}(\xi,\eta)\hat y_\nu(\tau,\xi-\eta)\hat y_\mu(\tau,\eta)\dd \eta \right]^{\tau=t}_{\tau=0}\\
        &\quad -\frac{1}{4\pi^2}\sum_{\nu,\mu\in\{\pm\}}\int_0^t\int_{\R^2}\frac1{\I\Phi_{\sigma,\nu\mu}(\xi,\eta)}\e^{\I\Phi_{\sigma,\nu\mu}(\xi,\eta) \tau} m_{\sigma,\nu\mu}(\xi,\eta)\frac{\dd}{\dd \tau}\left(\hat y_\nu(\tau,\xi-\eta)\hat y_\mu(\tau,\eta)\right)\dd \eta\dd \tau.
        \label{eq:normal_form_step1}
    \end{split}
    \end{align}
    Observe that
    \begin{equation}
        \e^{\I\Phi_{\sigma,\nu\mu}(\xi,\eta) \tau}\hat y_\nu(\tau,\xi-\eta)\hat y_\mu(\tau,\eta)=\e^{-\I\sigma\Omega(\xi)\tau}\hat z_\nu(\tau,\xi-\eta)\hat z_\mu(\tau,\eta),
    \end{equation}
    and that
    \begin{align}
        \frac{\dd}{\dd \tau}\bigg(\hspace{-1pt}\hat y_\nu(\tau,\xi-\eta)\hat y_\mu(\tau,\eta)\hspace{-1pt}\bigg)&=\frac{\dd}{\dd \tau}\bigg(\e^{-\I(\nu\Omega(\xi-\eta)+\mu\Omega(\eta))\tau}\hat z_\nu(\tau,\xi-\eta)\hat z_\mu(\tau,\eta)\bigg)\\
        &=\e^{-\I(\nu\Omega(\xi-\eta)+\mu\Omega(\eta))\tau}\hspace{-2pt}\left({\widehat{\mathcal{NL}}_\nu}(\tau,\xi-\eta)\hat z_\mu(\tau,\eta)+\hat z_\nu(\tau,\xi-\eta){\widehat{\mathcal{NL}}_\mu}(\tau,\eta)\right)\hspace{-2pt}.
    \end{align}
    These two properties, together with \eqref{eq:normal_form_step1}, conclude the proof.
\end{proof}
Given the expression for $z_\sigma$ in Lemma~\ref{lemma:expression_for_z_sigma}, a direct application of the linear dispersive estimate of Proposition~\ref{prop:general_dispersive_estimate} gives the following bound in homogeneous Besov spaces
\begin{equation}\label{eq:estimate_Besov_thm2}
\begin{split}
    \|z_\sigma(t)\|_{\dot B^1_{\infty,1}}&\lesssim\langle t\rangle^{-\frac12}\left(\norm{z_\sigma(0)}_{\dot B^3_{1,1}}+\sum_{\nu,\mu\in\{\pm\}}\norm{\cB_{\sigma,\nu\mu}(z_\nu,z_\mu)(0)}_{\dot B^3_{1,1}}\right)\\
    &\quad+\sum_{\nu,\mu\in\{\pm\}}\norm{\cB_{\sigma,\nu\mu}(z_\nu,z_\mu)(t)}_{\dot B^1_{\infty,1}}+\int_0^t\langle t-\tau\rangle^{-\frac12}\norm{\cE_\sigma(\tau)}_{\dot B^3_{1,1}}\dd\tau\\
    &\quad+\sum_{\nu,\mu\in\{\pm\}}\int_0^t\langle t-\tau\rangle^{-\frac12}\left(\norm{\cB_{\sigma,\nu\mu}(\mathcal{NL}_\nu,z_\mu)(\tau)}_{\dot B^3_{1,1}}+\norm{\cB_{\sigma,\nu\mu}(z_\nu,\mathcal{NL}_\mu)(\tau)}_{\dot B^3_{1,1}}\right)\dd \tau.
\end{split}
\end{equation}
It is now clear that, in order to improve the estimate \eqref{eq:estimate_to_be_improved_z_sigma} in the bootstrap argument, we need the pseudo-product operators $\cB_{\sigma,\nu\mu}$ to be bounded in terms of the Sobolev norms of their arguments.
This is the content of Lemma~\ref{lemma:estimate_pseudo_product} below, which is in the spirit of the product estimate in homogeneous Besov spaces of Proposition~\ref{prop:product_estim_Besov}.
Note that, compared with the estimate \eqref{eq:product_estim_Besov} for the plain product, the regularity required on the right-hand side of \eqref{eq:estimate_pseudo_product} is higher by one.
This comes from the fact that the multipliers $\{\cA^{2,j}_{\sigma,\nu\mu}\}$ appearing in the definition \eqref{eq:definition_m_symbols}  of $m_{\sigma,\nu\mu}$ are of order one.
\begin{lemma} \label{lemma:estimate_pseudo_product}
    Fix $k\in\NN\setminus\{0\}$.
    For all $g,h\in H^{s_0}$ with $s_0>k+1$, it holds
    \begin{equation} \label{eq:estimate_pseudo_product}
        \|\cB_{\sigma,\nu\mu}(g,h)\|_{\dot B^k_{1,1}} \lesssim_{N,f}\norm{g}_{H^{s_0}}\norm{h}_{H^{s_0}}.
    \end{equation}
\end{lemma}
Thanks to Lemma~\ref{lemma:estimate_pseudo_product}, whose proof is deferred to Section~\ref{sec:proof_lemmaT}, we are now ready to prove Proposition~\ref{prop:bootstrap} with $T_\eps\sim_{N,f}\varepsilon^{-2}$.
\begin{proof}[Proof of Proposition~\ref{prop:bootstrap}]
    Since the nonlinear term $J(a,b)$ on the right-hand side of \eqref{eq:D_bound_general} can be handled as in Lemma~\ref{lemma:control_of_J}, we focus on the oscillating variables $z_\pm$.
    Recalling that $\dot B^{k}_{1,1}\hookrightarrow\dot B^{k-2}_{\infty,1}$ for any $k\in\NN\setminus\{0\}$, from \eqref{eq:estimate_Besov_thm2} and Lemma~\ref{lemma:estimate_pseudo_product} we deduce that for any $s>5$ we have
    \begin{align}
    \begin{split}
        \|z_\sigma(t)\|_{\dot B^1_{\infty,1}}&\lesssim_{N,f}\langle t\rangle^{-\frac12}\left(\norm{z_\sigma(0)}_{\dot B^3_{1,1}}+\sum_{\nu,\mu\in\{\pm\}}\|z_\nu(0)\|_{H^{s-1}}\|z_\mu(0)\|_{H^{s-1}}\right)\\
        &\quad+\sum_{\nu,\mu\in\{\pm\}}\|z_\nu(t)\|_{H^{s-1}}\|z_\mu(t)\|_{H^{s-1}}+\int_0^t\langle t-\tau\rangle^{-\frac12} \norm{\cE_\sigma(\tau)}_{\dot B^3_{1,1}}\dd \tau\\
        &\quad+2\sum_{\nu,\mu\in\{\pm\}}\int_0^t\langle t-\tau\rangle^{-\frac12} \|\mathcal{NL}_\nu(\tau)\|_{H^{s-1}}\|z_\mu(\tau)\|_{H^{s-1}}\dd\tau.
        \label{eq:normal_form_step2}
    \end{split}
    \end{align}
    Arguing as in \eqref{eq:control_quadratic_nonlinearity}, we bound the nonlinearity as
    \begin{align}
    \begin{aligned}
        \norm{\cN\cL_\nu}_{H^{s-1}}&\lesssim_{N,f} \norm{U\omega}_{H^{s-1}}+\norm{Ua}_{H^{s}}+\norm{Ub}_{H^s}\\
        &\lesssim_{N,f} \norm{U}_{H^{s}}(\norm{\omega}_{H^{s-1}}+\norm{a}_{H^{s}}+\norm{b}_{H^{s}})\\
        &\lesssim_{N,f} \eps^2,
        \label{eq:normal_form_step3}
    \end{aligned}
    \end{align}
    and the error term as
    \begin{align}
    \begin{aligned}
        \norm{\cE_\sigma}_{\dot B^3_{1,1}}&\lesssim_{N,f}\sum_{j\in\{1,3\}}\norm{U\de_j\Omega^{-2}(D)|\nabla|^{-2}J(a,b)}_{\dot B^4_{1,1}}\\
        &\lesssim_{N,f}\norm{U}_{H^{s-1}}\norm{|\nabla|^{-1}J(a,b)}_{H^{s-1}}\\
        &\lesssim_{N,f}\varepsilon^3,
        \label{eq:normal_form_step4}
    \end{aligned}
    \end{align} 
    where we have used the commutator estimate in \cite[Corollary 5.4, Eq (5.8)]{Li2019} to bound
    \begin{align}
        \norm{|\nabla|^{-1}J(a,b)}_{H^{s-1}}=\|[\cJ_{s-1}\de_3|\nabla|^{-1},a]\de_1b-[\cJ_{s-1}\de_1|\nabla|^{-1},a]\de_3b\|_{L^2}\lesssim\|a\|_{H^{s-1}}\|b\|_{H^{s-1}}.
    \end{align}
    Putting together \eqref{eq:normal_form_step3} and \eqref{eq:normal_form_step4}, and using the assumption on the initial datum \eqref{eq:thm1_assumption_ic_0}, from \eqref{eq:normal_form_step2} we infer
    \begin{equation}
        \|z_\sigma(t)\|_{\dot B^1_{\infty,1}}\lesssim_{N,f}\langle t\rangle^{-\frac12}\varepsilon+\eps^2+\langle t\rangle^{\frac12}\varepsilon^3.
    \end{equation}
    This last estimate improves \eqref{eq:estimate_to_be_improved_z_sigma} in Section~\ref{sec:proofof43}, and yields, for $\eps$ small enough,
    \begin{equation}
        \int_0^{T_\eps}\left(\|z_-(t)\|_{\dot B^1_{\infty,1}}+\|z_+(t)\|_{\dot B^1_{\infty,1}}\right)\dd t\leq \frac{1}{4c}\log\left(\frac32\right),
    \end{equation}
    where $c$ is the constant appearing in \eqref{eq:blow-up_Uab}.
    This fact, together with Lemma~\ref{lemma:control_of_J} and the well-preparedness of the initial datum, implies \eqref{eq:proof_of_bootstrap_bu}.
    An application of the blow-up criterion in Proposition~\ref{prop:blow_up_criterion} then concludes the proof.
\end{proof}
\begin{remark}
    Under the \eqref{ass:gamma_wp} assumption, the error term $\cE$ features an additional contribution coming from the nonzero potential vorticity.
    However, in view of the discussion in Section~\ref{subsec:PV}, such additional term can be bounded arguing as in Lemma~\ref{lemma:control_of_J}, giving 
    \begin{equation}
        \norm{\cE^q_\sigma}_{\dot B^3_{1,1}}\lesssim_{N,f}\norm{U}_{H^s}\norm{q}_{H^{s-1}}\lesssim \eps^{1+\gamma}.
    \end{equation}
    Up to time $T_\eps\sim_{N,f}\tilde c_0\eps^{-2}$, its integrated dispersive contribution is proportional to $\tilde c_0^{3/2}\eps^{\gamma-2}$, which is small for $\gamma\geq2$.
\end{remark}

\subsection{Proof of Lemma~\ref{lemma:estimate_pseudo_product}}\label{sec:proof_lemmaT}

In view of the factorized structure \eqref{eq:definition_m_symbols} of the symbols $m_{\sigma,\nu\mu}$ and by linearity, the problem reduces to the study of the simpler, time-independent bilinear operators $\rmT_{\sigma,\nu\mu}$ defined as
\begin{equation}
    \rmT_{\sigma,\nu\mu}(g,h)(x)=\cF^{-1}\left(\int_{\R^2}\frac{\hat g(\xi-\eta)\hat h(\eta)}{\Phi_{\sigma,\nu\mu}(\xi,\eta)}\dd \eta\right)(x),\qquad \sigma,\nu,\mu\in\{\pm\},\quad g,h\in\cC^\infty_c(\R^2).
\end{equation}
Indeed, the pseudo-product operators $\cB_{\sigma,\nu\mu}$ can be expressed as
\begin{equation}\label{eq:B_intermsof_T}
        \cB_{\sigma,\nu\mu}(g,h)=-\frac{\I}{4\pi^2}\sum_{j=1}^\ell\cA^{0,j}_{\sigma,\nu\mu}(D)\rmT_{\sigma,\nu\mu}(\cA^{1,j}_{\sigma,\nu\mu}(D) g,\cA^{2,j}_{\sigma,\nu\mu}(D) h).
\end{equation}
The operators $\rmT_{\sigma,\nu\mu}$ satisfy the following bound.
\begin{lemma} \label{lemma:T}
    Fix $k\in\NN_0$ and let $s_0>k$.
    There exists a constant $C=C(\delta_*)>0$ such that, for all $g,h\in H^{s_0}$,
    \begin{equation}
        \norm{\rmT_{\sigma,\nu\mu}(g,h)}_{\dot B^k_{1,1}}\leq C\norm{g}_{H^{s_0}}\norm{h}_{H^{s_0}}.
    \end{equation}
\end{lemma}
\begin{proof}
    For any choice of $\sigma, \nu,\mu\in\{\pm\}$, let $\kappa$ be defined as in \eqref{eq:def_of_kappa}.
    Observe that, in view of Lemma~\ref{lemma:alg_Phi}, we have the representation
    \begin{equation}\label{eq:LaplacePhi}
        \frac{1}{\Phi_{\sigma,\nu\mu}(\xi,\eta)}=\kappa\int_0^{+\infty}\e^{-\kappa\Phi_{\sigma,\nu\mu}(\xi,\eta)\rho}\dd \rho,
    \end{equation}
    which is well defined since $\kappa \Phi_{\sigma,\nu\mu}(\xi,\eta)\geq\delta_*$.
    Now, \eqref{eq:LaplacePhi} allows us to write
    \begin{align}
        \rmT_{\sigma,\nu\mu}(g,h)&=\cF^{-1}\left(\kappa\int_{\R^2}\int_0^{\infty}\e^{\kappa\sigma\Omega(\xi)\rho}\e^{-\kappa\nu\Omega(\xi-\eta)\rho}\hat g(\xi-\eta)\e^{-\kappa\mu\Omega(\eta)\rho}\hat h(\eta)\dd\rho\dd \eta\right)\\
        &=\kappa\int_0^{+\infty}\e^{\kappa\sigma\Omega(D)\rho}\left[\left(\e^{-\kappa\nu\Omega(D)\rho} g\right)\left(\e^{-\kappa\mu\Omega(D)\rho} h\right)\right]\dd\rho,
    \end{align}
    where the exchange of integrals is justified by Fubini's theorem.
    Therefore, by the Young's inequality for convolutions, we have, for $\varphi\in \cC_c^\infty$ with $\supp \varphi \subset \{|\xi|\in[\frac34,\frac83]\}$,
    \begin{align}
    \begin{split}
        \|\rmT_{\sigma,\nu\mu}(g,h)\|_{\dot B^k_{1,1}}&\le\int_0^{\infty}\norm{\e^{\kappa\sigma\Omega(D)\rho}\left[\left(\e^{-\kappa\nu\Omega(D)\rho} g\right)\left(\e^{-\kappa\mu\Omega(D)\rho} h\right)\right]}_{\dot B^k_{1,1}}\dd\rho\\
        &\lesssim\int_0^{\infty}\norm{\mathcal{F}^{-1}\left(\varphi(\xi)\e^{\kappa\sigma\Omega(\xi)\rho}\right)}_{L^1}\norm{\left(\e^{-\kappa\nu\Omega(D)\rho} g\right)\left(\e^{-\kappa\mu\Omega(D)\rho} h\right)}_{\dot B^k_{1,1}}\dd\rho\\
        &\lesssim\int_0^{\infty}\norm{\mathcal{F}^{-1}\left(\varphi(\xi)\e^{\kappa\sigma\Omega(\xi)\rho}\right)}_{L^1}\norm{\e^{-\kappa\nu\Omega(D)\rho} g}_{H^{s_0}}\norm{\e^{-\kappa\mu\Omega(D)\rho} h}_{H^{s_0}}\dd\rho\\
        &\leq \|g\|_{H^{s_0}}\|h\|_{H^{s_0}} \int_0^\infty A_\sigma(\rho)B_\nu(\rho)B_\mu(\rho)\dd\rho,
        \label{eq:integral_to_bound_m1}
    \end{split}
    \end{align}
    where we used Proposition~\ref{prop:product_estim_Besov} to separate $g$ and $h$, and set
    \begin{equation}
        A_\sigma(\rho):=\norm{\mathcal{F}^{-1}\left(\varphi(\xi)\e^{\kappa\sigma\Omega(\xi)\rho}\right)}_{L^1},\quad B_\nu(\rho):=\norm{\e^{-\kappa\nu\Omega(\cdot)\rho}}_{L^\infty}.
    \end{equation}
    To conclude the proof, observe that the Plancherel identity and the fact that the support of $\varphi$ is compact yield
    \begin{align}
    \begin{split}
        A_\sigma(\rho)&\lesssim\norm{\langle x\rangle^2\mathcal{F}^{-1}\left(\varphi(\xi)\e^{\kappa\sigma\Omega(\xi)\rho}\right)(x)}_{L^2_x}=\norm{\langle\nabla_\xi\rangle^2\varphi(\xi)\e^{\kappa\sigma\Omega(\xi)\rho}}_{L^2_\xi}\\
        &\lesssim B_{-\sigma}(\rho)+\norm{\Delta_\xi\varphi(\xi)\e^{\kappa\sigma\Omega(\xi)\rho}}_{L^\infty_\xi}.
        \label{eq:integral_to_bound_1}
    \end{split}
    \end{align}
    Moreover, a direct computation shows that
    \begin{equation}
        |\Delta_\xi\varphi(\xi)\e^{\kappa\sigma\Omega(\xi)\rho}|\lesssim(1+\rho)^2B_{-\sigma}(\rho).
        \label{eq:integral_to_bound_2}
    \end{equation}
    Therefore, putting together \eqref{eq:integral_to_bound_1} and \eqref{eq:integral_to_bound_2}, we obtain that
    \begin{equation}
        \int_0^\infty A_\sigma(\rho)B_\nu(\rho)B_\mu(\rho)\dd\rho\lesssim\int_0^\infty(1+\rho)^2B_{-\sigma}(\rho)B_\nu(\rho)B_\mu(\rho)\dd\rho.
        \label{eq:integral_to_bound_3}
    \end{equation}
    Finally, observe that
    \begin{equation}
        \e^{-\kappa\nu\Omega(\xi)\rho}\le\begin{cases}
            \e^{\max\{|N|,|f|\}\rho}\quad&\text{if }\kappa\nu=-,\\
            \e^{-\min\{|N|,|f|\}\rho}\quad&\text{if }\kappa\nu=+,
        \end{cases}
    \end{equation}
    which implies
    \begin{equation}
        B_{-\sigma}(\rho)B_\nu(\rho)B_\mu(\rho)\le\e^{(\max\{|N|,|f|\}-2\min\{|N|,|f|\})\rho}=\e^{-\delta_*\rho}.
        \label{eq:integral_to_bound_4}
    \end{equation}
    Inserting \eqref{eq:integral_to_bound_4} into \eqref{eq:integral_to_bound_3} and recalling \eqref{eq:integral_to_bound_m1}, we conclude the proof, since $\delta_*>0$ ensures the integral is finite.
\end{proof}
\begin{proof}[Proof of Lemma~\ref{lemma:estimate_pseudo_product}]
    Combining \eqref{eq:B_intermsof_T} with the triangle inequality, Proposition~\ref{prop:homog_multi_appendix}, Corollary~\ref{coroll:0homo_multi}, and Lemma~\ref{lemma:T} concludes the proof.
\end{proof}

\addtocontents{toc}{\protect\setcounter{tocdepth}{-1}}

\section*{Acknowledgements}
The authors would like to thank Pierre Germain and Klaus Widmayer for insightful discussions during the preparation of this work.
The research of ADZ is supported by the SNSF through the grant PCEFP2 203059.
The research of LM is supported by the EPSRC under grant number EP/Y03533X/1, with additional support from the Department of Mathematics at Imperial College London.
The research of HX is supported by a Roth Scholarship at Imperial College London and the Royal Society Wolfson Fellowship awarded to Pierre Germain.

\section*{AI declaration}
AI tools were used in the preliminary phase of this work to identify papers in the literature related to the problem under consideration.
The representation \eqref{eq:LaplacePhi} was suggested by an AI tool, which ADZ was using to further understand the normal form theory and its application to the problem.
The authors then used this identity to develop the full argument presented in Lemma~\ref{lemma:T}.
Finally, AI was used to improve the language and exposition and to review the paper before finalizing the present version.

\addtocontents{toc}{\protect\setcounter{tocdepth}{1}}

\appendix

\section{Homogeneous Besov spaces}\label{app:appendixBesov}
We collect here the facts about homogeneous Besov spaces that are needed in this work, beginning with their definition.
A more comprehensive discussion can be found, for example, in \cite{BahCheDan2011}.

Let $\varphi\in \cC_c^\infty(\R^d)$ with $\supp\varphi \subset \{|\xi|\in[\frac34,\frac83]\}\subset\R^2$ such that the following partition of unity holds for all $\xi\neq0$
\begin{equation}
    \sum_{j\in\ZZ}\varphi(2^{-j}\xi)=1.
\end{equation}    
The homogeneous Littlewood--Paley decomposition of a distribution $u$ is then defined as
\begin{equation}
    u(x)=\sum_{j\in\ZZ}\dot \Delta_j u(x),
\end{equation}
with dyadic blocks
\begin{equation}
    \dot\Delta_j u(x)=\varphi(2^{-j}D)u(x)=\cF^{-1}(\varphi(2^{-j}\xi)\hat u(\xi))(x)=(\cF^{-1}(\varphi(2^{-j}\xi))*u)(x).
\end{equation}
\vspace{-16pt}
\begin{definition}[Homogeneous Besov spaces]
    Let $s\in\R$ and $p,q\in[1,\infty]$. We define $\dot B^s_{p,q}$ as the space of those (homogeneous) tempered distributions $u$ such that 
    \begin{equation}
        \norm{u}_{\dot B^s_{p,q}}:=\left(\sum_{j\in\ZZ}2^{jsq}\norm{\dot\Delta_ju}_{L^p}^q\right)^\frac1{q}<\infty.
    \end{equation}
    In the above, $s$ is the \textit{regularity} index, $p$ is the \textit{integrability} index, and $q$ is the \textit{summability} index.
\end{definition}
Homogeneous Besov spaces are the natural setting for studying the boundedness of operators associated with homogeneous Fourier symbols.
\begin{proposition}[Homogeneous multipliers {\cite[Proposition 2.30]{BahCheDan2011}}]\label{prop:homog_multi_appendix}
    Let $\cA$ be a smooth function, homogeneous of degree $k$, on $\R^d\setminus\{0\}$. 
    Then, if $(s_1,p_1,q_1)\in\R\times[1,\infty]^2$ is such that
    \begin{equation}
        s_1-k<\frac{d}{p_1}\quad \text{or}\quad s_1-k=\frac{d}{p_1} \text{ and } q_1=1,
        \label{eq:s1p1q1}
    \end{equation}
    the operator $\cA(D):\dot B^{s_1}_{p_1,q_1}\cap \dot B^{s_2}_{p_2,q_2}\longrightarrow \dot B^{s_1-k}_{p_1,q_1}\cap \dot B^{s_2-k}_{p_2,q_2}$ is continuous for all $(s_2,p_2,q_2)\in\R\times[1,\infty]^2$. 
\end{proposition}
Observe that if the triple $(s_1,p_1,q_1)$ satisfies \eqref{eq:s1p1q1}, then $\cA(D)u$ is a well-defined tempered distribution (see \cite[Lemma 2.23 and Remarks 2.24, 2.26]{BahCheDan2011}).
Once this is established, for instance setting $s_1=0$ and $p_1=q_1=2$ for $k=0$, we are free to choose the second triple $(s_2,p_2,q_2)\in\R\times[1,\infty]^2$.
This yields the following corollary.
\begin{corollary}[$0$-homogeneous multipliers]\label{coroll:0homo_multi}
    Let $\cA$ be a smooth $0$-homogeneous function on $\R^d\setminus\{0\}$. 
    Then, for $(s,p,q)\in\R\times[1,\infty]^2$, the operator $\cA(D):L^2\cap\dot B^s_{p,q}\to L^2\cap\dot B^{s}_{p,q}$ is continuous with norm bounded by 
    \begin{equation}\label{eq:general_bound_0homo_op}
        \norm{\cA(D)}_{\dot B^s_{p,q}\to \dot B^{s}_{p,q}}\lesssim \norm{\cF^{-1}(\varphi(\xi)\cA(\xi))}_{L^1}.
    \end{equation}
\end{corollary}
Applied to the multiplier $\Omega^{-1}$, the above corollary yields the following bound of its norm.
\begin{proposition}\label{prop:omega-1_bound}
    For any $(s,p,q)\in\R\times[1,\infty]^2$, the operator $\Omega^{-1}(D):\dot B^s_{p,q}\to \dot  B^s_{p,q}$ defined via the dispersion relation \eqref{eq:dispersion_relation_intro} is continuous with norm bound
    \begin{equation}\label{eq:Omega-1_bound}
        \norm{\Omega^{-1}(D)}_{\dot B^s_{p,q}\to \dot  B^s_{p,q}}\leq C\frac{\max\{|N|,|f|\}^2}{\min\{|N|,|f|\}^3},
    \end{equation} 
    for some constant $C=C(\varphi)$.
\end{proposition}
\begin{proof}
    Recall that 
    \begin{equation}
        \Omega^{-1}(\xi)=\sqrt{\frac{\xi^2_1+\xi^2_3}{N^2\xi^2_1+f^2\xi^2_3}},
    \end{equation}
    from which we obtain, for all $\xi\in\R^2\setminus\{0\}$,
    \begin{equation}\label{eq:pointwise_estim_omega-1}
    \begin{split}
        |\Omega^{-1}(\xi)|&\leq\frac{1}{\min\{|N|,|f|\}},\\
        |\nabla\Omega^{-1}(\xi)|&\lesssim \frac{1}{|\xi|}\frac{\max\{|N|,|f|\}}{\min\{|N|,|f|\}^2},\\
        |\Delta\Omega^{-1}(\xi)|&\lesssim \frac{1}{|\xi|^2}\frac{\max\{|N|,|f|\}^2}{\min\{|N|,|f|\}^3}.      
    \end{split}
    \end{equation}
    Then, equation \eqref{eq:Omega-1_bound} follows from \eqref{eq:general_bound_0homo_op} by combining a weighted Cauchy--Schwarz inequality with Plancherel's identity and \eqref{eq:pointwise_estim_omega-1}. 
    More precisely, we have
    \begin{align}
        \norm{\cF^{-1}(\varphi(\xi)\Omega^{-1}(\xi))}_{L^1}&\lesssim \norm{\langle \nabla_\xi\rangle^2\Omega^{-1}}_{L^2_\xi(\supp\varphi)}\\
        &\lesssim\left(\frac{1}{\min\{|N|,|f|\}}+\frac{\max\{|N|,|f|\}}{\min\{|N|,|f|\}^2}+\frac{\max\{|N|,|f|\}^2}{\min\{|N|,|f|\}^3}\right)\\
        &\lesssim \frac{\max\{|N|,|f|\}^2}{\min\{|N|,|f|\}^3}.
    \end{align}   
\end{proof}
\vspace{-5pt}
We conclude the appendix with a product estimate in $L^1$-based homogeneous Besov spaces.
\begin{proposition}[Product estimate]\label{prop:product_estim_Besov}
    Fix $k\in \NN\setminus\{0\}$.
    For all $g,h\in H^s$ with $s>k$, we have
    \begin{equation} \label{eq:product_estim_Besov}
        \norm{gh}_{\dot B^{k}_{1,1}}\lesssim \norm{g}_{H^s}\norm{h}_{H^s}.
    \end{equation}
\end{proposition}
\begin{proof}
    We can split the Besov norm into low and high frequencies contributions as follows
    \begin{equation}
        \norm{gh}_{\dot B^{k}_{1,1}}=\sum_{j=1}^\infty2^{-kj}\norm{\dot\Delta_{-j}(gh)}_{L^1}+\sum_{j=0}^\infty2^{kj}\norm{\dot\Delta_{j}(gh)}_{L^1}=:I_1+I_2.
        \label{eq:prod_estimate_step0}
    \end{equation}
    The low-frequency part can be bounded by means of Young's convolution inequality and the Cauchy--Schwarz inequality, recalling that $\|\cF^{-1}(\hat\varphi(2^{-j}\xi))\|_{L^1}=\|\cF^{-1}(\hat\varphi(\xi))\|_{L^1}$ for any $j\in\mathbb Z$:
    \begin{equation}
        I_1\le\sum_{j=1}^\infty2^{-kj}\|\cF^{-1}(\hat\varphi(\xi))\|_{L^1}\norm{g}_{L^2}\norm{h}_{L^2}\lesssim\norm{g}_{H^s}\norm{h}_{H^s}.
        \label{eq:prod_estimate_step1}
    \end{equation}
    For the high-frequency part, Bernstein's inequality allows us to argue as in \eqref{eq:prod_estimate_step1}:
    \begin{equation}
        I_2\lesssim\sum_{j=0}^\infty2^{j(k-s)}\norm{\dot\Delta_{j}|\nabla|^s(gh)}_{L^1}\lesssim\norm{|\nabla|^s(gh)}_{L^1}.
        \label{eq:prod_estimate_step2}
    \end{equation}
    Now, \cite[Theorem 1]{GraOh2014} gives the fractional Leibniz rule
    \begin{equation}
        \norm{|\nabla|^s(gh)}_{L^1}\lesssim\norm{|\nabla|^sg}_{L^2}\norm{h}_{L^2}+\norm{g}_{L^2}\norm{|\nabla|^sh}_{L^2},
    \end{equation}
    which, together with \eqref{eq:prod_estimate_step2}, gives
    \begin{equation}
        I_2\lesssim\norm{g}_{H^s}\norm{h}_{H^s}.
    \end{equation}
    In view of the decomposition \eqref{eq:prod_estimate_step0}, this concludes the proof.
\end{proof}
\bibliographystyle{abbrvnat}
\bibliography{ref}
\end{document}